\documentclass[a4paper,11pt,twoside]{article}
\usepackage[T1]{fontenc}
\usepackage[utf8]{inputenc}
\usepackage[english]{babel}
\usepackage{amsmath, amssymb, amsthm, amsfonts, mathtools, mathrsfs}	
\usepackage{enumerate}
\usepackage[shortlabels]{enumitem}
\usepackage{booktabs, caption, graphicx, subcaption, wrapfig} 	
\usepackage{geometry}
\usepackage{xcolor}

\usepackage{fancyhdr} 

\usepackage{cite} 
\usepackage{upgreek}
\usepackage{siunitx}
\usepackage{bbm}

\newcommand{\R}{\mathbb{R}}
\newcommand{\N}{\mathbb{N}}

\newcommand{\C}{\mathbb{C}}

\numberwithin{equation}{section}

\theoremstyle{plain}
\newtheorem{theorem}{Theorem}[section] 
\newtheorem{proposition}[theorem]{Proposition} 
\newtheorem{lemma}[theorem]{Lemma}
\newtheorem{corollary}[theorem]{Corollary}

\theoremstyle{definition}

\newtheorem{remark}[theorem]{Remark}

\DeclarePairedDelimiter{\abs}{\lvert}{\rvert}
\DeclarePairedDelimiter{\norm}{\lVert}{\rVert}
\DeclarePairedDelimiter{\duality}{\langle}{\rangle}

\DeclareMathOperator{\diag}{diag}

\newcommand{\eps}{\varepsilon}
\renewcommand{\phi}{\varphi}
\renewcommand{\bar}{\overline}
\renewcommand{\vec}{\boldsymbol}

\def\genspazio #1#2#3#4#5{#1^{#2}(#5,#4;#3)}
\def\spazio #1#2#3{\genspazio {#1}{#2}{#3}T0}

\def\LT {\spazio L}
\def\HT {\spazio H}
\def\C #1#2{\mathcal{C}^{#1}([0,T];#2)}

\def\Lx #1{L^{#1}(\Omega)}
\def\Lt #1{L^{#1}(0,T)}
\def\Lqt #1{L^{#1}(Q_T)}
\def\Hx #1{H^{#1}(\Omega)}

\def\Huo {H^1_0(\Omega)}

\def\ls{<}
\def\gs{>}
\def\mezzo {\frac{1}{2}}
\def\ddt {\frac{\de}{\de t}}
\def\de {\mathrm{d}}
\def\D {\mathrm{D}}

\def\n {\vec{n}}
\def\phib {\bar{\phi}}
\def\pb {\bar{p}}
\def\sigmab {\bar{\sigma}}
\def\hh {\mathbbm{h}}

\def\Iad {\mathcal{I}_{\text{ad}}}
\def\XX {\mathbb{X}}
\def\YY {\mathbb{Y}}

\def\HH {\mathbb{H}}
\def\VV {\mathbb{V}}
\def\vpsi {\vec{\uppsi}}
\def\vD {\vec{D}}
\def\vy {\vec{Y}}
\def\vh {\vec{h}}
\def\tb {\bar{t}}
\def\ts {t^\star}
\def\radice {\sqrt{\abs{\log \eps}}}
\def\Cb {\bar{C}}
\def\Qb {\bar{Q}}
\def\Rcal {\mathcal{R}}

\def\triplein {(\phi_0, \sigma_0, p_0)}

\def\Ccal {\mathcal{C}}

\def\phimeas {\phi_{\text{meas}}}
\def\sigmameas {\sigma_{\text{meas}}}
\def\pmeas  {p_{\text{meas}}}

\usepackage{hyperref}

\begin{document}
	
	\begin{center}
		

            \Large{\textbf{Mathematical analysis of a model-constrained inverse problem for the reconstruction of early states of prostate cancer growth }}  
		
		\vskip0.35cm
		
		\large{\textsc{Elena Beretta$^1$}} \\
		\normalsize{e-mail: \texttt{eb147@nyu.edu}} \\
		\vskip0.35cm
		
		\large{\textsc{Cecilia Cavaterra$^2$}} \\
		\normalsize{e-mail: \texttt{cecilia.cavaterra@unimi.it}} \\
		\vskip0.35cm
		
		\large{\textsc{Matteo Fornoni$^3$}}\\
		\normalsize{e-mail: \texttt{matteo.fornoni01@universitadipavia.it}} \\
		\vskip0.35cm

       \large{\textsc{Guillermo Lorenzo$^4$}} \\
		\normalsize{e-mail: \texttt{guillermo.lorenzo.gomez@sergas.es}} \\
		\vskip0.35cm
		
		\large{\textsc{Elisabetta Rocca$^5$}} \\
		\normalsize{e-mail: \texttt{elisabetta.rocca@unipv.it}} \\
		\vskip0.35cm
		
		\normalsize{$^1$Division of Science, NYU Abu Dhabi}
		\vskip0.1cm
		
		\normalsize{$^2$Department of Mathematics ``F. Enriques'', University of Milan \& IMATI-C.N.R.}
		\vskip0.1cm
		
		\normalsize{$^3$Department of Mathematics ``F. Casorati'', University of Pavia}
		\vskip0.1cm

        \normalsize{$^4$Health Research Institute of Santiago de Compostela, Spain \& Oden Institute for Computational Engineering and Sciences, The University of Texas at Austin, USA}
	\vskip0.1cm
	
        \normalsize{$^5$Department of Mathematics ``F. Casorati'', University of Pavia \& IMATI-C.N.R.}
		\vskip0.5cm
		
	\end{center}

	\begin{abstract}\noindent
         The availability of cancer measurements over time enables the personalised assessment of tumour growth and therapeutic response dynamics. 
         However, many tumours are treated after diagnosis without collecting longitudinal data, and cancer monitoring protocols may include infrequent measurements. 
         To facilitate the estimation of disease dynamics and better guide ensuing clinical decisions, we investigate an inverse problem enabling the reconstruction of earlier tumour states by using a single spatial tumour dataset and a biomathematical model describing disease dynamics. 
         We focus on prostate cancer, since aggressive cases of this disease are usually treated after diagnosis.
         We describe tumour dynamics with a phase-field model driven by a generic nutrient ruled by reaction-diffusion dynamics. 
         The model is completed with another reaction-diffusion equation for the local production of prostate-specific antigen, which is a key prostate cancer biomarker.
         We first improve previous well-posedness results by further showing that the solution operator is continuously Fréchet differentiable. 
         We then analyse the backward inverse problem concerning the reconstruction of earlier tumour states starting from measurements of the model variables at the final time. 
         Since this problem is severely ill-posed, only very weak conditional stability of logarithmic type can be recovered from the terminal data. 
         However, by restricting the unknowns to a compact subset of a finite-dimensional subspace, we can derive an optimal Lipschitz stability estimate. 
         
		\vskip3mm
		
		\noindent {\bf Key words:} phase field; non-linear parabolic system; inverse problems; well-posedness; mathematical oncology; prostate cancer.
		
		\vskip3mm
		
		\noindent {\bf AMS (MOS) Subject Classification: 
        35K51, 
        35K58 
        35R30,  
        35Q92, 
        92C50 
        } 
		
	\end{abstract}


\pagestyle{fancy}
\fancyhf{}	
\fancyhead[EL]{\thepage}
\fancyhead[ER]{\textsc{Beretta -- Cavaterra -- Fornoni -- Lorenzo -- Rocca}} 
\fancyhead[OL]{\textsc{Reconstructing early states of prostate cancer}} 
\fancyhead[OR]{\thepage}

\renewcommand{\headrulewidth}{0pt}
\setlength{\headheight}{5mm}

\thispagestyle{empty} 

\section{Introduction}

The mathematical modelling of cancer growth and therapeutic response has enabled the investigation of the biophysical mechanisms underlying the development and treatment of these diseases, as well as the prediction of personalised outcomes to guide clinical decision-making \cite{yin2019review,Lorenzo2022_review}.
Personalised tumour forecasting requires patient-specific measurements of cancer growth and treatment response, such as biomarker, imaging, histopathological and omics data \cite{Lorenzo2022_review,kazerouni2020integrating}.
These measurements enable the calibration of model parameters, the definition of the tumour and host tissue geometry, and the initialisation of model variables (e.g., those describing tumour and other cell species as well as key substances controlling their dynamics) \cite{yin2019review,Lorenzo2022_review,kazerouni2020integrating}.
Towards these ends, data are usually required at two or more time points during the monitoring of the disease before, during and after treatment \cite{Lorenzo2022_review,Lorenzo2024,Hormuth2021,wu2022mri}.
The collection of these data over time contributes to early characterise crucial clinical endpoints (e.g., progression to higher-risk disease, treatment failure, survival) by both analysing the relative change of these measurements \cite{giganti2021natural,tudorica2016early} as well as by using them to obtain more accurate tumour forecasts  \cite{Lorenzo2022_review,Lorenzo2024,Hormuth2021,wu2022mri}. 
However,  depending on the natural history and specific clinical management protocols for each type of cancer, it may not be possible to obtain tumour measurements until the tumour has developed sufficiently to produce symptoms or be detected with standard-of-care screening methods \cite{Lorenzo2022_review,kazerouni2020integrating,chaudhuri2023predictive,Lorenzo2024}.
Furthermore, tumour monitoring strategies may not enable to assess the tumour status as often as necessary to capture tumour growth and treatment response in detail \cite{Lorenzo2022_review,kazerouni2020integrating,chaudhuri2023predictive,Lorenzo2024}.
These situations may lead to limited accuracy in the estimation of disease prognosis and result in treatment excesses and deficiencies that can respectively affect the patients’ quality of life and life expectancy \cite{gupta2022systemic,ziu2020role,neal2020ten}.

\begin{figure}[t]
    \centering
    \includegraphics[width=1\linewidth]{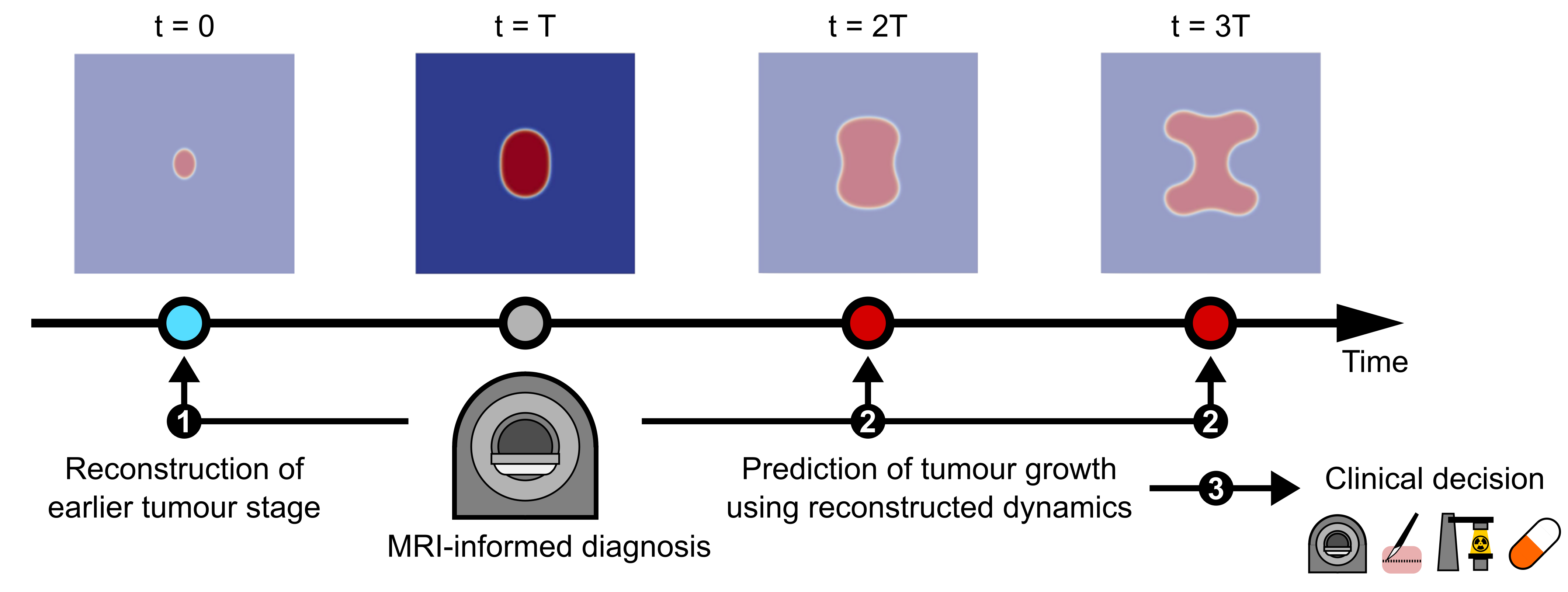}
    \caption{\textbf{Reconstruction of earlier prostate cancer states for clinical decision-making.} This figure illustrates the potential application of reconstructed prostate cancer dynamics from a diagnostic medical imaging dataset (e.g., MRI), while using a mathematical model describing prostate cancer growth. The first step (1) consists of estimating an earlier tumour state for the tumour within a clinically reasonable time, for example, six to twelve months in the context of monitoring prostate cancer or decision of first-line treatment. This is the step in which we focus on the present work. Then, a second step (2) leverages the reconstructed tumour growth dynamics to make a forecast within comparable time horizons to the one used in the reconstruction of earlier states. In the third and final step (3), these predictions are analysed to inform clinical decisions, for instance, whether to schedule another medical imaging scan to monitor lower-risk disease or to prescribe immediate treatment due to anticipated worsening of the disease in the short-term future (e.g., surgery, radiotherapy, or drug-based therapies).}
    \label{fig:globscheme}
\end{figure}

After the acquisition of a medical imaging measurement to characterise the tumour morphology (e.g., at diagnosis, after treatment), an inverse problem leveraging an adequate mathematical model to represent tumour dynamics and aiming at identifying the tumour status at an earlier date can shed light on the patient-specific tumour dynamics \cite{Lorenzo2022_review,SSMB2020,chaudhuri2023predictive}.
This reconstruction can then be used to better estimate potential clinical outcomes that guide crucial decision-making by leveraging tumour forecasts or model-based biomarkers.
Figure~\ref{fig:globscheme} illustrates the overall procedure whereby a spatial measurement of a tumour obtained \emph{via} medical imaging at the time of diagnosis is used to reconstruct an earlier tumour state and better inform the ensuing management of the disease.
Indeed, the determination of the spatial region where a tumour originates can be of clinical relevance for some tumour types, such as brain and prostate cancers \cite{ali2024tale, jungk2019location}.
Additionally, the need to estimate earlier states of tumour development can also arise in a pre-clinical context. For example, it may be necessary to determine the configuration of an \emph{in vitro} cell colony or the tumour morphology in an \emph{in vivo} animal model before the onset of experimental measurements or due to sparse data collection with important tumour dynamics between consecutive data points.

In this work, we analyse a model-constrained inverse problem aiming to reconstruct earlier tumour configurations from a spatial measurement collected, for instance, \emph{via} medical imaging at the time of diagnosis.
We focus on newly detected prostate cancer.
The diagnosis of these tumours relies on a multiparametric magnetic resonance imaging (MRI) scan, which is often motivated by increasing values of the blood biomarker prostate-specific antigen (PSA; the most common biomarker in prostate cancer management) and guides the ensuing biopsy procedure to confirm the disease histopathologically \cite{mottet2021eau,Lorenzo2024}.
The first-line clinical management options for prostate cancer usually include active surveillance for indolent, lower-risk cases and radical treatment (e.g., surgery, radiotherapy with/without androgen deprivation therapy) for aggressive, higher-risk patients \cite{mottet2021eau}.
Thus, the accurate identification of the clinical risk of prostate cancer is a critical need of utmost importance to select an adequate clinical management strategy.
Beyond the diagnostic PSA, MRI, and biopsy data, estimating the recent dynamics of prostate cancer growth before diagnosis can further inform the treating physicians about the best course of management for each particular patient \cite{giganti2021natural,Lorenzo2024}.
Towards this end, we employ a mathematical model of prostate cancer growth based on the phase-field method, which has been studied both from an analytical and a numerical point of view in \cite{CGLMRR2019} and \cite{CGLMRR2021}.
This PDE modelling approach is a common spatiotemporal continuous formulation to describe the time evolution of the location and geometry of a tumour, as well as to address related optimal control problems 
\cite{Cavaterra2019, Colli2015, Colli2017, F2024, Frigeri2015, FLS2021, Garcke2016, Garcke2018, Lorenzo2016, Miranville2019, Frieboes2010, Wise2008, Xu2016}. 
Hence, we define a continuous phase field $\phi$ that takes values $\phi\approx0$ in the host tissue, $\phi\approx1$ in the tumour, and intermediate values in a small neighbourhood of the tumour-healthy tissue interface, which exhibits a smooth yet steep profile.
The tumour phase field dynamics depends on the local availability of a critical nutrient (e.g.,~oxygen or glucose), whose concentration is denoted by $\sigma$ and obeys a reaction-diffusion equation.
Additionally, the model includes another reaction-diffusion equation to represent the local production of tissue PSA $p$, which represents the PSA leaked to the bloodstream per unit volume of prostatic tissue \cite{Lorenzo2016}.
The complete prostate cancer model considered herein reads as follows: 
\begin{alignat}{2}
    & \partial_t \phi - \lambda \Delta \phi + F'(\phi) - m(\sigma) \hh'(\phi) = 0 
    && \hbox{in $Q_T$,} \label{eq:phi0} \\
    & \partial_t \sigma - \eta \Delta \sigma 
    = S_h (1-\phi) + S_c \phi 
    - \left( \gamma_h (1-\phi) + \gamma_c \phi  \right) \sigma
    \qquad && \hbox{in $Q_T$,} \label{eq:sigma0} \\
    & \partial_t p - D \Delta p + \gamma_p p 
    = \alpha_h (1 - \phi)+ \alpha_c \phi 
    && \hbox{in $Q_T$,} \label{eq:p0} \\
    & \phi = 0, \quad \partial_{\n} \sigma = 0, \quad \partial_{\n} p = 0 
    && \hbox{on $\Sigma_T$,} \label{bc0} \\
    & \phi(0) = \phi_0, \quad \sigma(0) = \sigma_0, \quad p(0) = p_0
    && \hbox{in $\Omega$,} \label{ic0}
\end{alignat}
where $\Omega \subset \R^N$, $N=2,3$, is a sufficiently regular open and bounded domain with outward unit normal vector $\n$, $T \gs 0$ is the final time (or time horizon), $Q_T := \Omega \times (0,T)$, and $\Sigma_T = \partial \Omega \times (0,T)$. 
For the inverse problem addressed herein, $t=T$ represents the time of the diagnostic imaging scan while $t=0$ is a previous timepoint at which we want to reconstruct the tumour phase field.

The diffusion coefficient of the tumour phase in Eq.~\eqref{eq:phi0} can be defined as $\lambda=M\ell^2$, where both $M$ and $\ell$ are positive real constants denoting the tumour mobility and interface width \cite{Xu2016}.
The non-linear functions $F$ and $\hh$ are explicitly defined as:
\begin{align*}
	F(\phi) = M \phi^2 (1-\phi)^2, \qquad \hh(\phi) = M \phi^2 (3 -2 \phi), \qquad \text{with } M \gs 0.
\end{align*}
The non-convex function $F(\phi)$ is a double-well potential, which is typical in phase field modelling and allows for the coexistence of the tumour ($\phi\approx1$) and healthy ($\phi\approx0$) tissues.  
The interpolation function $\hh(\phi)$ is also common in non-conserved dynamics of phase fields and it verifies the properties $\hh(0)=0$ and $\hh(1)=1$ as well as $\hh^\prime(0)=\hh^\prime(1)=0$, where $\hh^\prime$ denotes the derivative of $\hh$. 
The function $m(\sigma)$ is usually called tilting function and, in our model, is defined as
\begin{equation}\label{msigma}
m(\sigma )=m_{ref}\left( \frac{\rho +A}{2}+\frac{\rho -A}{\pi }\arctan \left( \frac{\sigma -\sigma _{l}}{\sigma _{r}}\right) \right),
\end{equation}
where $m_{ref}$ is a positive constant that scales the strength of the tilting function within our phase field framework, while $\rho$ and $A$ are constant proliferation and apoptosis indices, respectively. These non-dimensional parameters are defined upon the proliferation and apoptosis rates in tumour tissue as follows:
\begin{equation*}
\rho=\frac{K_\rho}{\bar{K_\rho}}, 
\quad \text{and} \quad 
A=-\frac{K_A}{\bar{K_A}},
\end{equation*}
where $K_\rho$ is the proliferation rate of tumour cells, $K_A$ is the apoptosis rate of tumour cells, $\bar{K_\rho}$ is a scaling reference value for the proliferation rate, and $\bar{K_A}$  is a scaling reference value for the apoptosis rate. 
Thus, according to the definition of $m(\sigma)$ provided above, this function represents the nutrient-dependent net proliferation rate of the tumour.
Additionally, the constants $\sigma_r$ and $\sigma_l$ in Eq.~\eqref{msigma} denote a reference and a threshold value for the nutrient concentration\cite{Xu2016}.
Therefore, nutrient concentrations lower than $\sigma_l$ render the healthy tissue energetically more favorable than the tumour tissue and vice-versa. 
Moreover, when $|m(\sigma)|<1/3$, the function $G(\phi,\sigma)=F(\phi)-m(\sigma)\hh(\phi)$ is still a double-well potential presenting two local minima at $\phi=0$ and $\phi=1$ \cite{CGLMRR2019}.
Within this range, low nutrient concentrations result in a lower energy level (i.e., the value of $G$) in the healthy tissue ($\phi=0$) than in the tumoral tissue ($\phi=1$). The opposite is true if we consider high values of the nutrient concentration $\sigma$. The described behaviours of $m(\sigma)$ and $G(\phi,\sigma)$ are further illustrated in Figure \ref{functionplots}.

\begin{figure}[t]
    \begin{subfigure}[t]{0.32\textwidth}
      \includegraphics[width=\textwidth]{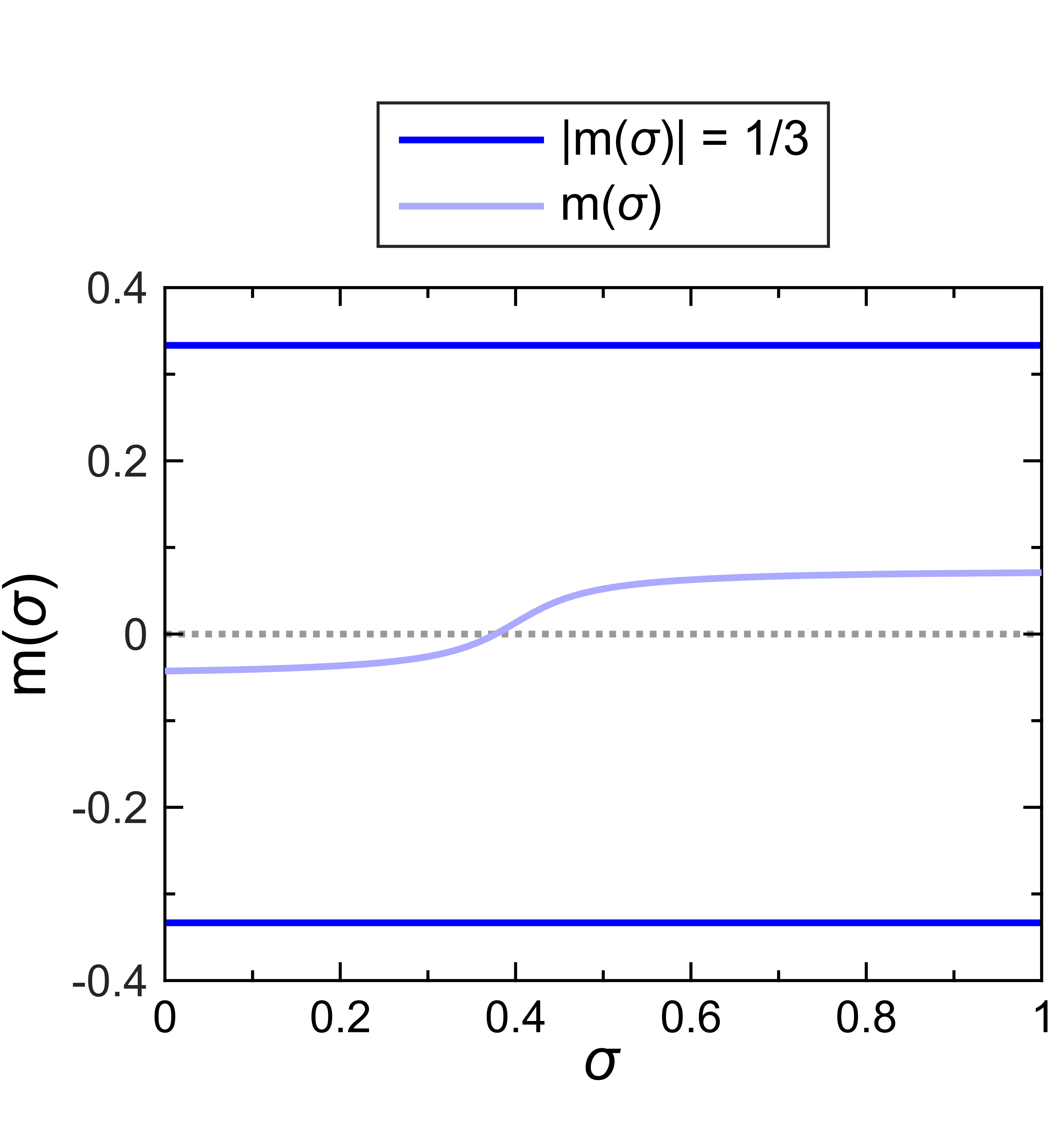}
      \caption{}
    \end{subfigure}
    \hfill
    \begin{subfigure}[t]{0.32\textwidth}
      \includegraphics[width=\textwidth]{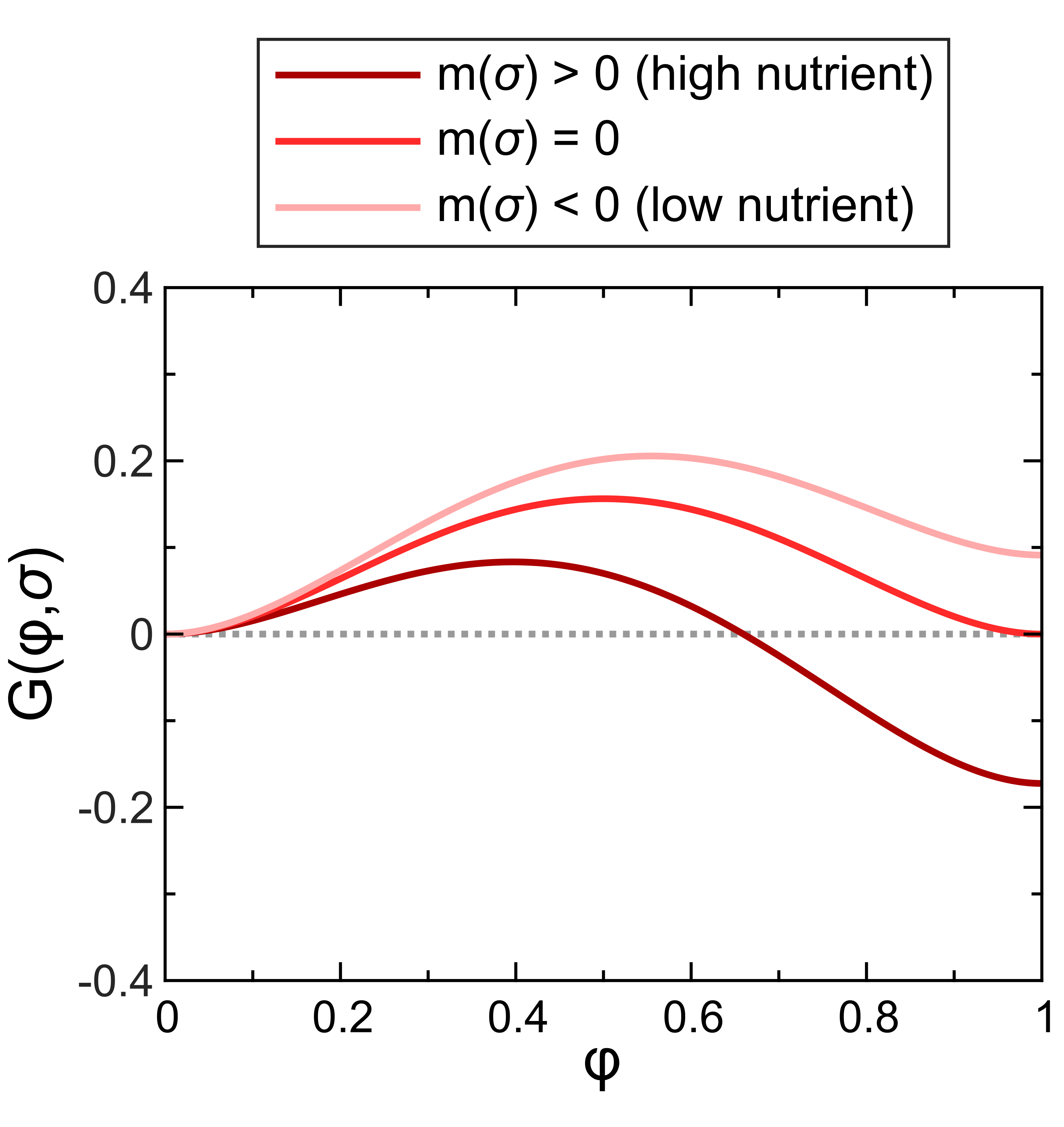}
      \caption{}
    \end{subfigure}
    \hfill
    \begin{subfigure}[t]{0.32\textwidth}
      \includegraphics[width=\textwidth]{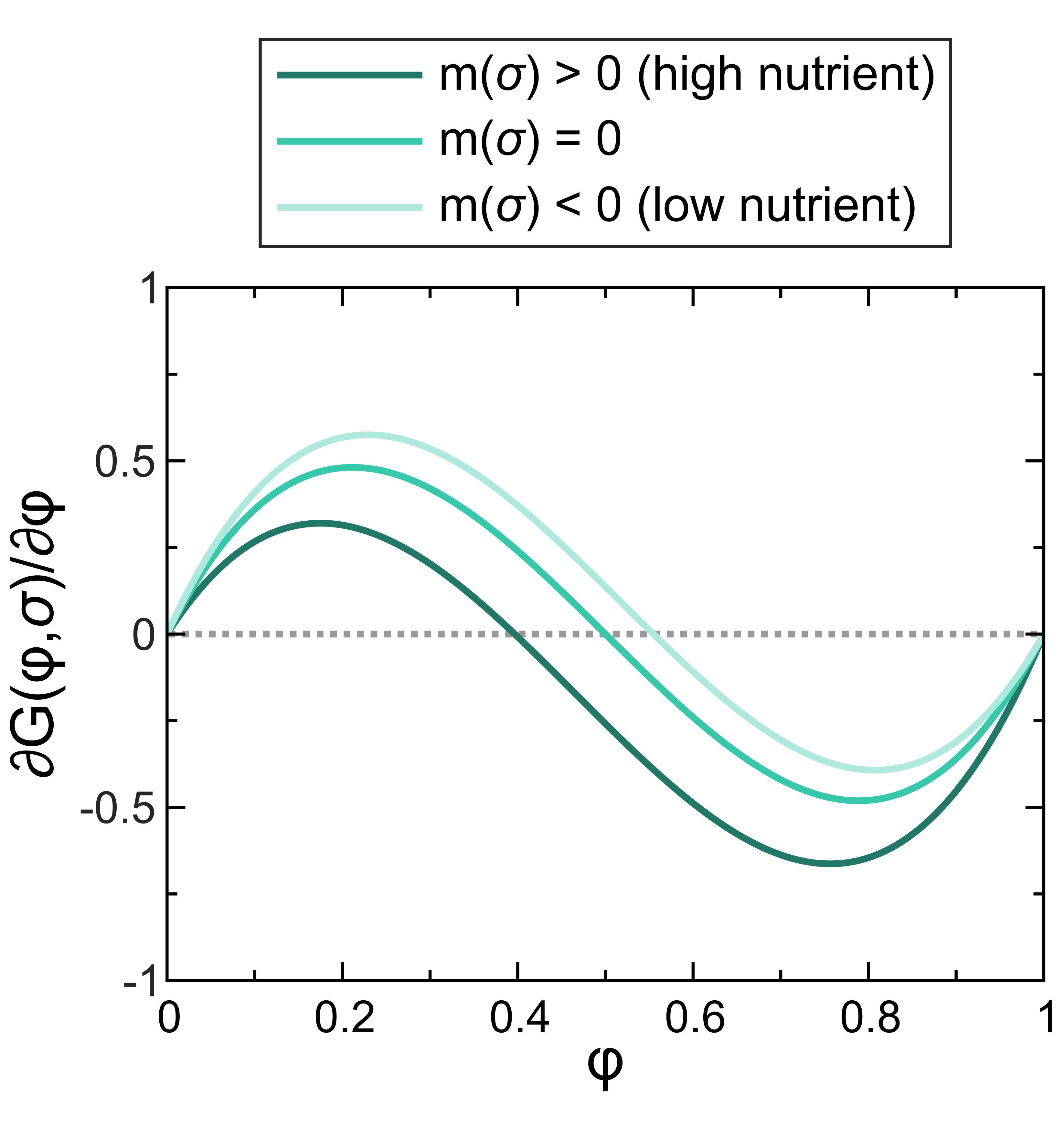}
      \caption{}
    \end{subfigure}
    \hfill
\vspace*{8pt}
\caption{\textbf{Nonlinear terms in the phase field model of prostate cancer growth.}
Panel (a) depicts the tilting function $m(\sigma)$, which represents the net proliferation, with respect to the local nutrient $\sigma$.
This plot shows that proliferation is supported in regions with higher nutrient concentrations, while tumour cell death dominates in areas where the nutrient is scarcer.
Panel (b) shows the potential $G\left(\phi,\sigma\right)= F(\phi) - m(\sigma) \hh(\phi)$ with respect to the values of the tumour phase field $\phi$ in three situations: a high nutrient availability that fosters proliferation ($m(\sigma)>0)$, limited nutrient availability that inhibits proliferation ($m(\sigma)=0$), and a scarce nutrient concentration that promotes tumour cell death ($m(\sigma)<0$). 
For the same three scenarios, panel (c) further provides the derivative $\partial G\left(\phi,\sigma\right)/\partial\phi = F'(\phi) - m(\sigma) \hh'(\phi) $, which explicitly drives the spatiotemporal dynamics of the tumour phase field in Eq.~\eqref{eq:phi}.
The plots in panels (b) and (c) show that a high nutrient concentration energetically favours tumour growth within our modelling framework, whereas a low nutrient environment energetically impedes it.}
\label{functionplots}
\end{figure}

In the nutrient dynamics equation (Eq.~\eqref{eq:sigma0}), $\eta$ denotes the diffusion coefficient of the nutrient, $S_h$ and $S_c$ stand for the non-negative nutrient supply rates in the healthy and tumour tissue, and $\gamma_h$ and $\gamma_c$ are positive constants that represent the nutrient uptake rate in the healthy and cancerous tissue.  
Additionally, the parameters governing the tissue PSA equation (Eq.~\ref{eq:p0}) are the diffusion coefficient $D$, the rates of production of tissue PSA by the healthy tissue $\alpha_h$ and the tumour $\alpha_c$, and the natural decay rate $\gamma_p$. 
To conclude the exposition of the model, we choose standard no-flux boundary conditions for $\sigma$ and $p$, while we set zero-valued Dirichlet boundary conditions for $\phi$.
This last assumption implies that the tumour is confined inside the prostate (recall indeed that $\phi \approx 0$ represents the healthy phase), which is the case of the majority of newly-diagnosed prostate cancer cases \cite{mottet2021eau}.
Furthermore, the only difference with respect to the original model presented in \cite{CGLMRR2019} is that we have neglected the terms related to cytotoxic and antiangiogenic treatments.
The rationale for this choice is that here we focus on reconstructing early stages of prostate cancer growth before the time of first diagnosis using imaging data and, thus, before the administration of any treatment \cite{Lorenzo2024,mottet2021eau}. 
The interested reader is referred to \cite{CGLMRR2019,CGLMRR2021} for a more detailed presentation of the model equations, parameters, and alternative definitions (e.g., further mechanisms to include in the model, spatially-varying parameters, nonlinearities).

Mathematically, the problem we want to address is a backward inverse problem for a nonlinear system of parabolic equations.
Such kind of problems arise any time a diffusion process has to be reversed thereby governing a multitude of applications \cite{Arratia2021, JBJA2019, SSMB2020, Engl1995, Akcelik2006}. 
To be precise, we call $\mathcal{R}: (\phi_0, \sigma_0, p_0) \mapsto (\phi(T), \sigma(T), p(T))$ the nonlinear solution operator which associates to any initial data $(\phi_0, \sigma_0, p_0)$ the value of the solution at the final time $(\phi(T), \sigma(T), p(T))$. 
Then, given some measurements $\phimeas, \sigmameas, \pmeas: \Omega \to \R$, we can formulate our inverse problem as that of finding initial data $\phi_0$, $\sigma_0$, $p_0$ such that the solution $(\phi, \sigma, p)$ to \eqref{eq:phi0}--\eqref{ic0} satisfies $\phi(T) = \phimeas$, $\sigma(T) = \sigmameas$ and $p(T) = \pmeas$. 
The backward inverse problem, even in the linear case, is well-known to be severely ill-posed with very weak conditional logarithmic stability from the data \cite{Payne75, LRS, isakov}. 
Therefore, as a preliminary step, it is of paramount importance to regularise the problem,  by establishing physically relevant a priori assumptions on the unknown initial data that lead to better, possibly  Lipschitz, dependence of the initial data on the measurements. 
In turn, such stability estimates give local convergence of iterative reconstruction methods \cite{DQS2012}. 
This strategy has been successfully employed in addressing other non-linear severely ill-posed problems, for example in the context of the conductivity inverse problem, see for example \cite{AV2005, BF2011, ABFV2022}. 
The key idea is that, by restricting the unknowns to a compact subset of a finite-dimensional subspace, we can derive Lipschitz dependence, via a quantitative version of the inverse map theorem in Banach spaces \cite[Proposition 5]{BV2006}.
To accomplish this, we need to prove the injectivity and continuity of the inverse map, as well as Fr\'echet differentiability with continuity of the forward map. 
Indeed, even though the system is non-linear, we can adapt logarithmic convexity methods to prove injectivity and a conditional stability estimate for the inverse map for any positive time $t \gs 0$ (Proposition \ref{prop:condstab}). Additionally, through a more refined analysis, we are also able to recover a logarithmic stability estimate for the reconstruction of the initial data (Theorem \ref{thm:inizstab}).
We just mention that the log-convexity method was already employed for example in \cite{Hao2011} to derive stability estimates for linear parabolic partial differential equations and in \cite{K2007} in the case of non-linear parabolic equations. 
Carleman estimates for parabolic equations can be also employed to derive uniqueness and stability estimates for the backward inverse problem, see for instance \cite{Y2009, IY1998, IY2014, Vessella2009}.
Then, to prove the additional regularity of $\Rcal$ required by \cite[Proposition 5]{BV2006}, we show finer continuous dependence estimates on the solutions of the system \eqref{eq:phi0}--\eqref{ic0}, that allow us to further get that $\Rcal$ is $\Ccal^1$ (Theorem \ref{thm:frechet}).
As a consequence, we can derive the sought quantitative Lipschitz stability estimate (Theorem \ref{thm:lipstab}), under the a priori assumption that the initial data lie in a compact subset of a finite-dimensional subspace.
In particular, our explicit computation of the stability constants confirms that, as expected, the inverse problem is severely ill-posed. 
Indeed, such constants blow up exponentially with the dimension of the finite-dimensional subspace and doubly exponentially with the final time $T$.
However, such a Lipschitz stability estimate paves the way for the use of iterative locally convergent reconstruction algorithms, such as the Landweber iteration method \cite{DQS2012, kaltenbacher:neubauer:scherzer}.
Due to the challenging nature of the ill-posed inverse problem, we plan to address an in-depth numerical study in a forthcoming paper, aiming to explore the possibility of an efficient reconstruction to be eventually used in clinical procedures.

The paper is organised as follows. 
In Section~\ref{sec2} we recall the well-posedness results shown in \cite{CGLMRR2019} and then prove that the solution operator is continuously Fr\'echet differentiable. 
Section~\ref{sec3} is devoted to the analysis of the inverse problem of reconstructing the initial data.
Here we prove our main stability results in Theorem \ref{thm:inizstab} and Theorem \ref{thm:lipstab}.
Section~\ref{sec:conclusion} finally summarises the obtained results and explores future directions for the numerical study of the problem.

\section{Well-posedness and differentiability of the solution operator}
\label{sec2}

We first introduce some notation that will be used throughout the paper. Let $\Omega \subset \R^N$, $N=2,3$, be an open and bounded domain with $\mathcal{C}^2$ boundary and outward unit normal vector $\n$. Let $T \gs 0$ be the final time, and denote $Q_t := \Omega \times (0,t)$ and $\Sigma_t = \partial \Omega \times (0,t)$ for any $t \in (0,T]$. 
We define the spaces:
\[ H = \Lx2, \quad V_0 = H^1_0(\Omega), \quad V_0^* = (H^1_0(\Omega))^*, \quad V = \Hx1, \quad V^* = (\Hx1)^*, \]
\[ W_0 = \Hx2 \cap \Huo, \quad W = \left\{ u \in \Hx2 \mid \partial_{\n} u = 0 \right\}. \]
By standard results, we know that, if $H$ is identified with its dual, the following compact and dense embeddings hold:
\[ W_0 \hookrightarrow V_0 \hookrightarrow H \hookrightarrow V_0^* \quad \text{and} \quad W \hookrightarrow V \hookrightarrow H \hookrightarrow V^*.  \]
Recall that, by elliptic regularity, we can use the equivalent norms:
\[ \norm{u}^2_{W_0} = \norm{u}^2_H + \norm{\Delta u}^2_H, \quad \norm{u}^2_W = \norm{u}^2_H + \norm{\Delta u}^2_H. \]
In some cases, we will denote for simplicity:
\[ \HH = H \times H \times H, \quad \VV = V_0 \times V \times V. \]
By using standard Bochner theory for spaces of time-dependent functions with values on Banach spaces, we also introduce the spaces in which weak and strong solutions of system \eqref{eq:phi}--\eqref{ic} will be respectively defined:
\begin{align*}
    \YY_0 & = \HT1{V_0^*} \cap \C0H \cap \LT2{V_0}, \\
    \YY & = \HT1{V^*} \cap \C0H \cap \LT2V, \\
    \XX_0 & = \HT1H \cap \C0{V_0} \cap \LT2{W_0},  \\
    \XX & = \HT1H \cap \C0V \cap \LT2W. 
\end{align*}
Regarding the parameters of the system, we assume the following hypotheses:
\begin{enumerate}[font = \bfseries, label = A\arabic*., ref = \bf{A\arabic*}]
	\item\label{ass:coeff} $\lambda, \eta, \gamma_h, \gamma_c, S_h, S_c, D, \gamma_p, \alpha_h, \alpha_c > 0$.
    \item\label{ass:Fh} $F(s) =  M s^2(1-s)^2$ and $\hh(s) = M s^2(3-2s)$, with $M>0$, for any $s \in \R$. In particular, we observe that $F, \hh \in \Ccal^3(\R)$. 
    \item\label{ass:m} $m(s) = m_{\text{ref}} \left( \frac{\rho + A}{2} + \frac{\rho - A}{2} \arctan \left( \frac{s - \sigma_l}{\sigma_r} \right) \right)$, with $m_{\text{ref}}, \rho, A, \sigma_l, \sigma_r > 0$. In particular, we observe that $m$ and $m'$ are Lipschitz continuous on $\R$ and $m, m', m'' \in L^\infty(\R)$.
\end{enumerate}
In what follows, we will extensively use the symbol $C \gs 0$, which may also change from line to line, to denote positive constants depending only on the fixed parameters of the system and possibly on $\Omega$ and $T$. In some cases, we will use a subscript to highlight some particular dependence of these constants.
Finally, to simplify a bit the analysis of the source terms in the model, we rewrite it in the following way:
\begin{alignat}{2}
	& \partial_t \phi - \lambda \Delta \phi + F'(\phi) - m(\sigma) \hh'(\phi) = 0 
	&& \hbox{in $Q_T$,} \label{eq:phi} \\
	& \partial_t \sigma - \eta \Delta \sigma = S_h + S_{ch} \phi - \gamma_h \sigma - \gamma_{ch} \sigma \phi 
	\qquad && \hbox{in $Q_T$,} \label{eq:sigma} \\
	& \partial_t p - D \Delta p + \gamma_p p 
	= \alpha_h + \alpha_{ch} \phi 
	&& \hbox{in $Q_T$,} \label{eq:p} \\
	& \phi = 0, \quad \partial_{\n} \sigma = \partial_{\n} p = 0 
	&& \hbox{in $\Sigma_T$,} \label{bc} \\
	& \phi(0) = \phi_0, \quad \sigma(0) = \sigma_0, \quad p(0) = p_0
	&& \hbox{in $\Omega$,} \label{ic}
\end{alignat}
where we call $\gamma_{ch} := \gamma_c - \gamma_h$, $S_{ch} := S_c - S_h$ and $\alpha_{ch} := \alpha_c - \alpha_h$.

Our first step consists of the analysis of the forward system.
We start by recalling the following result about the well-posedness of system \eqref{eq:phi}--\eqref{ic}, proved in \cite[Theorem 3.2]{CGLMRR2019}. 

\begin{theorem}
\label{thm:wellposedness}
    Assume hypotheses \ref{ass:coeff}--\ref{ass:m} and let 
    \begin{equation}
    \label{ass:initial_phi}
        (\phi_0, \sigma_0, p_0) \in \HH, \quad \text{and} \quad 0 \le \phi_0(x) \le 1 \quad \text{for a.e. } x \in \Omega.
    \end{equation}
    Then, there exists a unique weak solution 
    \[ (\phi, \sigma, p) \in \YY_0 \times \YY \times \YY, \]
    with 
    \begin{equation}
    \label{phi:bounded}
        0 \le \phi(x,t) \le 1 \quad \text{a.e. } (x,t) \in Q_T,
    \end{equation}
    which solves system \eqref{eq:phi}--\eqref{ic} in variational formulation, namely
    \begin{align*}
        & \duality{\partial_t \phi, v}_{V_0} + \lambda (\nabla \phi, \nabla v)_H + (F'(\phi) - m(\sigma) \hh'(\phi), v)_H = 0 
    	&& \hbox{a.e. in $(0,T)$, $\forall v \in V_0$}, \\
    	& \duality{\partial_t \sigma, v}_V + \eta (\nabla \sigma, \nabla v)_H + (\gamma_h \sigma + \gamma_{ch} \sigma \phi 
    	- S_h - S_{ch} \phi, v)_H = 0 
    	&& \hbox{a.e. in $(0,T)$, $\forall v \in V$}, \\
    	& \duality{\partial_t p, v}_V + D (\nabla p, \nabla v)_H + (\gamma_p p - \alpha_h - \alpha_{ch} \phi, v)_H = 0
    	&& \hbox{a.e. in $(0,T)$, $\forall v \in V$},
    \end{align*}
    with $\phi(0)=\phi_0$, $\sigma(0)=\sigma_0$ and $p(0)=p_0$, and satisfies 
    \begin{equation}
    \label{weaksols:est}
        \norm{\phi}^2_{\YY_0} + \norm{\sigma}^2_{\YY} + \norm{p}^2_{\YY} \le C \left( 1 + \norm{\phi_0}^2_H + \norm{\sigma_0}^2_H + \norm{p_0}^2_H \right),
    \end{equation}
    for some $C \gs 0$ depending only on the parameters of the system.
    Moreover, if 
    \begin{equation}
    \label{ass:initial_sigmap}
        \begin{split}
            \sigma_0, \, p_0 \in \Lx \infty, \quad \text{and} \quad \sigma_0(x) \ge 0, \, p_0(x) \ge 0 \text{ for a.e. } x \in \Omega,
        \end{split}
    \end{equation}
    then also 
    \[ \sigma, p \in \Lqt \infty, \quad \text{and} \quad \sigma(x,t) \ge 0, \, p(x,t) \ge 0 \text{ for a.e. } (x,t) \in Q_T, \]
    with 
    \[ \norm{\sigma}_{\Lqt\infty} \le C (\norm{\sigma_0}_{\Lx\infty} + 1), \quad \norm{p}_{\Lqt\infty} \le C (\norm{p_0}_{\Lx\infty} + 1). \]
    Furthermore, the solution is continuous with respect to the initial data, that is, taking two solutions $(\phi_i, \sigma_i, p_i)$ corresponding to $(\phi_0^i, \sigma_0^i, p_0^i)$, $i =1,2$, we have 
    \begin{equation}
        \label{contdep:est}
        \begin{split}
        & \norm{(\phi_1 - \phi_2)(t)}^2_{H} + \norm{(\sigma_1 - \sigma_2)(t)}^2_{H} + \norm{(p_1 - p_2)(t)}^2_{H} \\
        & \quad + \norm{\phi_1 - \phi_2}^2_{\LT 2 {V_0}} + \norm{\sigma_1 - \sigma_2}^2_{\LT 2 V} + \norm{p_1 - p_2}^2_{\LT 2 V} \\ 
        & \qquad \le C \left( \norm{\phi_0^1 - \phi_0^2}^2_H + \norm{\sigma_0^1 - \sigma_0^2}^2_H + \norm{p_0^1 - p_0^2}^2_H \right),
        \end{split}
    \end{equation}
    for any $t \in [0,T]$. Finally, if $(\phi_0, \sigma_0, p_0) \in \VV$, the solution enjoys the higher regularity $(\phi, \sigma, p) \in \XX_0 \times \XX \times \XX$, solves the system \eqref{eq:phi}--\eqref{ic} almost everywhere and satisfies the estimate
    \begin{equation}
    \label{strongsols:est}
        \norm{\phi}^2_{\XX_0} + \norm{\sigma}^2_{\XX} + \norm{p}^2_{\XX} \le C \left( 1 + \norm{\phi_0}^2_{V_0} + \norm{\sigma_0}^2_V + \norm{p_0}^2_V \right).
    \end{equation}
\end{theorem}

\begin{remark}
    Due to our model, hypotheses \ref{ass:Fh} and \ref{ass:m} are very specific on the non-linearities that we consider.
    However, from an analytical point of view, one could consider more general expressions of $F$, $\hh$ and $m$, as long as the results of Theorem \ref{thm:wellposedness} still hold.
    Hence, the forthcoming analysis can easily be extended to a more general framework.
\end{remark}

We additionally prove a stronger continuous dependence estimate, which will be useful when proving the Fr\'echet-differentiability of the solution operator. 

\begin{proposition}
    \label{prop:strongcontdep}
    Assume hypotheses \ref{ass:coeff}--\ref{ass:m} and let $(\phi_0^i, \sigma_0^i, p_0^i) \in \VV$, $i =1,2$, be two sets of initial data, satisfying also \eqref{ass:initial_phi} and \eqref{ass:initial_sigmap}. Let $(\phi_i, \sigma_i, p_i)$, $i=1,2$, be the corresponding strong solutions of \eqref{eq:phi}--\eqref{ic}. Then, there exists a constant $C > 0$, depending only on the parameters of the system and on $(\phi_0^i, \sigma_0^i, p_0^i)$, $i =1,2$, but not on their difference, such that
    \begin{equation}
        \label{contdep:eststrong}
        \begin{split}
        & \norm{(\phi_1 - \phi_2)(t)}^2_{V_0} + \norm{(\sigma_1 - \sigma_2)(t)}^2_{V} + \norm{(p_1 - p_2)(t)}^2_{V} \\
        & \quad + \norm{\phi_1 - \phi_2}^2_{\LT 2 {W_0}} + \norm{\sigma_1 - \sigma_2}^2_{\LT 2 W} + \norm{p_1 - p_2}^2_{\LT 2 W} \\ 
        & \qquad \le C \left( \norm{\phi_0^1 - \phi_0^2}^2_{V_0} + \norm{\sigma_0^1 - \sigma_0^2}^2_V + \norm{p_0^1 - p_0^2}^2_V \right), \quad  \forall \, t \in [0,T].
        \end{split}
    \end{equation}
\end{proposition}

\begin{proof}
Let us define the differences
    \[ \phi := \phi_1 - \phi_2, \quad \sigma := \sigma_1 - \sigma_2, \quad p := p_1 - p_2. \]
    Then, up to adding and subtracting some terms, they satisfy the following system of equations:
    \begin{alignat}{2}
		& \partial_t \phi - \lambda \Delta \phi + ( F'(\phi_1) - F'(\phi_2) ) \nonumber \\ 
        & \qquad - m(\sigma_1) (\hh'(\phi_1) - \hh'(\phi_2)) - (m(\sigma_1) - m(\sigma_2)) \hh'(\phi_2) = 0
		&& \hbox{in $Q_T$,} \label{eq:phicont} \\
		& \partial_t \sigma - \eta \Delta \sigma + \gamma_h \sigma + \gamma_{ch} \sigma_1 (\phi_1 - \phi_2) + \gamma_{ch} (\sigma_1 - \sigma_2) \phi_2 = S_{ch} \phi
		\qquad && \hbox{in $Q_T$,} \label{eq:sigmacont} \\
		& \partial_t p - D \Delta p + \gamma_p p = \alpha_{ch} \phi
		&& \hbox{in $Q_T$,} \label{eq:pcont} \\
		& \phi = 0, \quad \partial_{\n} \sigma = \partial_{\n} p = 0 
		&& \hbox{in $\Sigma_T$,} \label{bccont} \\
		& \phi(0) = \phi_0^1 - \phi_0^2, \quad \sigma(0) = \sigma_0^1 - \sigma_0^2, \quad p(0) = p_0^1 - p_0^2
		&& \hbox{in $\Omega$.} \label{iccont}
	\end{alignat}
	Note that, being the initial data in $\VV$, by Theorem \ref{thm:wellposedness} the equalities above have to be intended in a strong sense. Therefore, we are allowed to test \eqref{eq:phicont} by $- \Delta \phi$, \eqref{eq:sigmacont} by $-\Delta \sigma$, \eqref{eq:pcont} by $- \Delta p$ and sum them up, to obtain:
    \begin{equation}
    \label{contdep:mainest}
        \begin{split}
            & \mezzo \ddt \left( \norm{\nabla \phi}^2_H + \norm{\nabla \sigma}^2_H + \norm{\nabla p}^2_H \right) + \lambda \norm{\Delta \phi}^2_H + \eta \norm{\Delta \sigma}^2_H + D \norm{\Delta p}^2_H \\
            & \quad = - \gamma_h \norm{\nabla \sigma}^2_H - \gamma_p \norm{\nabla p}^2_H +  (F'(\phi_1) - F'(\phi_2), \Delta \phi)_H \\
            & \qquad - (m(\sigma_1) (\hh'(\phi_1) - \hh'(\phi_2)), \Delta \phi)_H - ((m(\sigma_1) - m(\sigma_2)) \hh'(\phi_2), \Delta \phi)_H \\
            & \qquad + \gamma_{ch} (\sigma_1 (\phi_1 - \phi_2), \Delta \sigma)_H + \gamma_{ch} ((\sigma_1 - \sigma_2) \phi_2, \Delta \sigma)_H \\
            & \qquad - S_{ch} (\phi, \Delta \sigma)_H - \alpha_{ch} (\phi, \Delta p)_H.
        \end{split}
    \end{equation}
    Then, we can easily estimate the terms on the right-hand side by using Cauchy-Schwarz and Young's inequalities, together with the local Lipschitz continuity of $F'$, $\hh'$, $m$ and the fact that $\phi_i, \sigma_i \in \Lqt\infty$ for $i=1,2$. Indeed, we have that: 
    \begin{align*}
        & (F'(\phi_1) - F'(\phi_2), \Delta \phi)_H - (m(\sigma_1) (\hh'(\phi_1) - \hh'(\phi_2)), \Delta \phi)_H - ((m(\sigma_1) - m(\sigma_2)) \hh'(\phi_2), \Delta \phi)_H \\
        & \qquad + \gamma_{ch} (\sigma_1 (\phi_1 - \phi_2), \Delta \sigma)_H + \gamma_{ch} ((\sigma_1 - \sigma_2) \phi_2, \Delta \sigma)_H - S_{ch} (\phi, \Delta \sigma)_H - \alpha_{ch} (\phi, \Delta p)_H \\
        & \quad \le \frac{\lambda}{2} \norm{\Delta \phi}^2_H + \frac{\eta}{2} \norm{\Delta \sigma}^2_H + \frac{D}{2} \norm{\Delta p}^2_H + C \norm{\phi}^2_H + C \norm{\sigma}^2_H. 
    \end{align*}
    Hence, by also integrating on $(0,t)$, for any $t \in (0,T)$, we get that 
    \begin{align*}
        & \mezzo \left( \norm{\nabla \phi(t)}^2_H + \norm{\nabla \sigma(t)}^2_H + \norm{\nabla p(t)}^2_H \right) + \frac{\lambda}{2} \int_0^t \norm{\Delta \phi}^2_H \, \de s + \frac{\eta}{2} \int_0^t \norm{\Delta \sigma}^2_H \, \de s + \frac{D}{2} \int_0^t \norm{\Delta p}^2_H \, \de s \\
        & \quad \le \mezzo \left( \norm{\phi_0^1 - \phi_0^2}^2_{V_0} + \norm{\sigma_0^1 - \sigma_0^2}^2_V + \norm{p_0^1 - p_0^2}^2_V \right) + C \int_0^T \norm{\phi}^2_{V_0} \, \de s + C \int_0^T \norm{\sigma}^2_V \, \de s,
    \end{align*}
    which by Gronwall's inequality and \eqref{contdep:est} implies \eqref{contdep:eststrong}. Proposition \ref{prop:strongcontdep}
    is proved.
\end{proof}

Given the results of Theorem \ref{thm:wellposedness} and Proposition \ref{prop:strongcontdep}, we introduce the set of admissible initial data as follows:
\begin{equation}
    \label{admset:initial}
    \Iad = \left\{ (\phi_0, \sigma_0, p_0) \in \VV \mid 0 \le \phi_0 \le 1, \, 0 \le \sigma_0 \le \sigma_{\text{max}}, \, 0 \le p_0 \le p_{\text{max}} \right\},
\end{equation}
where $\sigma_{\text{max}}, p_{\text{max}} \in \Lx \infty$ are given. Note that $\Iad$ is a closed and convex subset of $\VV \cap \Lx\infty^3$.
Next, we want to study in detail the operator $\Rcal: \Iad \to \HH$ which associates to any initial data $(\phi_0, \sigma_0, p_0) \in \Iad$ the corresponding solution $(\phi(T), \sigma(T), p(T))$ to \eqref{eq:phi}--\eqref{ic}, evaluated at the final time, that is 
\begin{equation}
    \label{def:R}
    \Rcal: \Iad \to \HH, \quad \Rcal((\phi_0, \sigma_0, p_0)) = (\phi(T), \sigma(T), p(T)).
\end{equation}
Note that $\Rcal$ is well-defined and Lipschitz-continuous by Theorem \ref{thm:wellposedness}. 
We now want to prove that it is also continuously Fr\'echet-differentiable. To do this, we introduce the linearised system:
\begin{alignat}{2}
	& \partial_t Y - \lambda \Delta Y + F''(\phib) Y - m(\sigmab) \hh''(\phib) Y - m'(\sigmab) \hh'(\phib) Z = 0 
	\qquad && \hbox{in $Q_T$,} \label{eq:philin} \\
	& \partial_t Z - \eta \Delta Z + \gamma_h Z + \gamma_{ch} \sigmab Y + \gamma_{ch} \phib Z - S_{ch} Y = 0 
	&& \hbox{in $Q_T$,} \label{eq:sigmalin} \\
	& \partial_t P - D \Delta P + \gamma_p P 
	= \alpha_{ch} Y 
	&& \hbox{in $Q_T$,} \label{eq:plin} \\
	& Y = 0, \quad \partial_{\n} Z = \partial_{\n} P = 0 
	&& \hbox{on $\Sigma_T$,} \label{bclin} \\
	& Y(0) = h, \quad Z(0) = k, \quad P(0) = w
	&& \hbox{in $\Omega$,} \label{iclin}
\end{alignat}
where $h, k ,w \in \Lx2$. This system is deduced by fixing an initial datum $(\phib_0, \sigmab_0, \pb_0)$, as well as its corresponding solution $(\phib, \sigmab, \pb)$, by linearising near it, i.e.~by considering 
\[ \phi = \phib + Y, \, \sigma = \sigmab + Z, \, p = \pb + P, \, \phi_0 = \phib_0 + h, \, \sigma_0 = \sigmab_0 + k, \, p_0 = \pb_0 + w, \]
and by recovering the system solved by $(Y,Z,P)$, while approximating the non-linearities at the first-order of their Taylor expansion. We have the following result about the well-posedness of \eqref{eq:philin}--\eqref{iclin}. 

\begin{proposition}
\label{prop:linearised}
    Assume hypotheses \ref{ass:coeff}--\ref{ass:m} and let $(h, k, w) \in \HH$. Let $(\phib_0, \sigmab_0, \pb_0) \in \Iad$. Then, \eqref{eq:philin}--\eqref{iclin} admits a unique weak solution $(Y,Z,P) \in \YY_0 \times \YY \times \YY$, which solves the system in variational formulation and satisfies the estimate:
    \[ \norm{Y}^2_{\YY_0} + \norm{Z}^2_{\YY} + \norm{P}^2_{\YY} \le C \left( \norm{h}^2_H + \norm{k}^2_H + \norm{w}^2_H \right). \]
    Moreover, if $(h, k, w) \in \VV$, the solution enjoys higher regularity $(Y, Z, P) \in \XX_0 \times \XX \times \XX$, solves the system almost everywhere and satisfies the estimate:
    \[ \norm{Y}^2_{\XX_0} + \norm{Z}^2_{\XX} + \norm{P}^2_{\XX} \le C \left(\norm{h}^2_{V_0} + \norm{k}^2_V + \norm{w}^2_V \right). \]
\end{proposition}

\begin{proof}
    Note that system \eqref{eq:philin}--\eqref{iclin} is composed of three \emph{linear} parabolic equations without sources and with coefficients in $\Lqt \infty$, thanks to Theorem \ref{thm:wellposedness}. Then, the statements of Proposition \ref{prop:linearised} follow by standard results from the theory of parabolic equations.
\end{proof}

The linearised system \eqref{eq:philin}--\eqref{iclin} is an \emph{ansatz} for the explicit expression of the Fr\'echet-derivative of the operator $\Rcal$. Indeed, we are now in position to show that 
\begin{equation}
    \label{frechet:derivative}
    \D\Rcal(\phib_0, \sigmab_0, \pb_0)[(h,k,w)] = (Y(T), Z(T), P(T)).
\end{equation}
Let $\mathcal{I}_R$ be an open subset of $\VV \cap \Lx\infty^3$ such that $\Iad \subseteq \mathcal{I}_R$. Note that, in $\mathcal{I}_R$, the continuous dependence estimate \eqref{contdep:eststrong} holds with $C$ depending only on $R$ and the fixed data of the system.
Our aim is to show that $\Rcal: \mathcal{I}_R \to \HH$ is continuously Fr\'echet-differentiable.
Indeed, we can prove:

\begin{theorem}
    \label{thm:frechet}
    Assume hypothesis \ref{ass:coeff}--\ref{ass:m}. Then $\Rcal: \mathcal{I}_R \to \HH$ is Fr\'echet-differentiable, i.e. for any $(\phib_0, \sigmab_0, \pb_0) \in \mathcal{I}_R$ there exists a unique Fr\'echet-derivative $\D\Rcal((\phib_0, \sigmab_0, \pb_0)) \in \mathcal{L}(\VV \cap L^\infty(\Omega)^3, \HH)$ such that, as $\norm{(h,k,w)}_{\VV} \to 0$,
	\begin{equation}
		\label{frechet:diff}
		\frac{ \norm{ \Rcal((\phib_0 + h, \sigmab_0 + k, \pb_0 + w)) - \mathcal{R}((\phib_0, \sigmab_0, \pb_0)) - \D\mathcal{R}((\phib_0, \sigmab_0, \pb_0))[(h,k,w)]  }_{\HH}}{ \norm{(h,k,w)}_{\VV} } \to 0.
	\end{equation}
	Moreover, for any $(h,k,w) \in \VV \cap \Lx\infty^3$, the Fr\'echet-derivative at $(\phib_0, \sigmab_0, \pb_0)$ in $(h,k,w)$ is defined as in \eqref{frechet:derivative}, where $(Y(T), Z(T), P(T))$ is the solution to the linearised system \eqref{eq:philin}--\eqref{iclin} with initial data $(h,k,w)$, evaluated at the final time. 

    Additionally, the Fr\'echet-derivative $\D\Rcal$ is Lipschitz-continuous as a function from $\mathcal{I}_R$ to the space $\mathcal{L}(\VV \cap L^\infty(\Omega)^3, \HH)$. More precisely, the following estimate holds:
    \begin{equation}
        \label{lipconst:frechet}
        \norm{\D\mathcal{R}((\phib_0^1, \sigmab_0^1, \pb_0^1)) - \D\mathcal{R}((\phib_0^2, \sigmab_0^2, \pb_0^2))}^2_{\mathcal{L}(\VV \cap \Lx\infty^3, \HH)} \le C_0 \norm{(\phib_0^1, \sigmab_0^1, \pb_0^1) - (\phib_0^2, \sigmab_0^2, \pb_0^2)}^2_{\HH},
    \end{equation}
    for some constant $C_0 \gs 0$ depending only on the parameters of the system.
\end{theorem}

\begin{proof}
    We observe that it is sufficient to prove the result for any small enough perturbation $(h,k,w)$, i.e., we fix $\Lambda > 0$ and consider only perturbations such that
	\begin{equation}
		\label{hkw:bound}
		\norm{(h,k,w)}_{\Lx\infty^3} \le \Lambda.
	\end{equation} 
	Now, we fix $(\phib_0, \sigmab_0, \pb_0)$ and $(h,k,w)$ as above and consider
	\begin{align*}
		& (\phi, \sigma, p) \text{ as the solution to \eqref{eq:phi}--\eqref{ic} with data } (\phib_0 + h, \sigmab_0 + k, \pb_0 + w), \\
		& (\phib, \sigmab, \pb) \text{ as the solution to \eqref{eq:phi}--\eqref{ic} with data } (\phib_0, \sigmab_0, \pb_0), \\
		& (Y, Z, P) \text{ as the solution to \eqref{eq:philin}--\eqref{iclin} with data } (h,k,w).
	\end{align*}
    For the Fr\'echet-differentiability, it is enough to show that there exist a constant $C>0$, depending only on the parameters of the system and possibly on $\Lambda$, and an exponent $s > 2$ such that
	\[ \norm{ (\phi(T), \sigma(T), p(T)) - (\phib(T), \sigmab(T), \pb(T)) - (Y(T), Z(T), P(T)) }^2_{\HH} \le C \norm{(h,k,w)}^s_{\HH}. \]
    To do this, we introduce the additional variables
	\begin{align*}
		& \psi := \phi - \phib - Y \in \XX_0, \\
		& \theta := \sigma - \sigmab - Z \in \XX, \\
        & \zeta  := p - \pb - P \in \XX,
	\end{align*}
	which, by Theorem \ref{thm:wellposedness} and Proposition \ref{prop:linearised}, enjoy the regularities shown above. Then, this is equivalent to showing that
	\begin{equation}
		\label{frechet:aim}
		\norm{ (\psi(T), \theta(T), \zeta(T)) }_{\HH}^2 \le C \norm{(h,k,w)}^s_{\HH}.
	\end{equation}
    By inserting the equations solved by the variables in the definitions of $\psi$, $\theta$ and $\zeta$ and exploiting the linearity of the involved differential operators, we infer that $\psi$, $\theta$ and $\zeta$ satisfy:
	\begin{alignat}{2}
		& \partial_t \psi - \lambda \Delta \psi + F^h = M^h  && \text{in } Q_T,  \label{eq:phidiff}\\
		& \partial_t \theta - \eta \Delta \theta + \gamma_h \theta = Q^h + S_{ch} \psi   \qquad && \text{in } Q_T, \label{eq:sigmadiff} \\
        & \partial_t \zeta - D \Delta \zeta + \gamma_p \zeta = \alpha_{ch} \psi && \text{in } Q_T,  \label{eq:pdiff}
	\end{alignat}
	together with boundary and initial conditions:
	\begin{alignat}{2}
		& \psi = 0, \quad \partial_{\n} \theta = \partial_{\n} \zeta = 0 \qquad && \text{on } \Sigma_T, \label{bcdiff} \\
		& \psi(0) = 0, \quad \theta(0) = 0, \quad \zeta(0) = 0 \qquad && \text{in } \Omega, \label{icdiff}
	\end{alignat}
	where
	\begin{align*}
		F^h & = F'(\phi) - F'(\phib) - F''(\phib)Y, \\
        M^h & = m(\sigma) \hh'(\phi) - m(\sigmab) \hh'(\phib) - m(\sigmab) \hh''(\phib) Y - m'(\sigmab) \hh'(\phib) Z, \\
        Q^h & = \gamma_{ch} \sigma \phi - \gamma_{ch} \sigmab \phib - \gamma_{ch} \sigmab Y - \gamma_{ch} \phib Z.
	\end{align*} 
    Before going on, we can rewrite in a better way the terms $F^h$, $M^h$ and $Q^h$, by using the following version of Taylor's theorem with integral remainder for a real function $f \in \mathcal{C}^2$ at a point $x_0 \in \R$:
	\[ f(x) = f(x_0) + f'(x_0) (x-x_0) + \left( \int_0^1 (1-z) f''(x_0 + z(x-x_0)) \, \de z \right) (x-x_0)^2. \]
	Indeed, with straightforward calculations, one can see that
	\begin{align*}
		F^h = F''(\phib) \psi + R_1^h (\phi - \phib)^2,
	\end{align*}
	and also, up to adding and subtracting some additional terms, that
	\begin{align*}
        M^h & = (m(\sigma) - m(\sigmab))(\hh'(\phi)-\hh'(\phib)) + m(\sigmab)[\hh''(\phib) \psi + R_2^h (\phi - \phib)^2] \\
        & \quad + \hh'(\phi) [m'(\sigmab) \theta + R_3^h (\sigma - \sigmab)^2], \\
		Q^h & = \gamma_{ch} (\sigma - \sigmab)(\phi - \phib) + \gamma_{ch} \sigmab \psi + \gamma_{ch} \phib \theta,
	\end{align*}
	where
	\begin{gather*}
		R_1^h = \int_0^1 (1-z) F'''(\phib + z(\phi - \phib)) \, \de z, \quad R_2^h = \int_0^1 (1-z) \hh'''(\phib + z(\phi - \phib)) \, \de z, \\
        R_3^h = \int_0^1 (1-z) m''(\sigmab + z(\sigma - \sigmab)) \, \de z.
	\end{gather*}
    By using global boundedness of the variables, given by Theorem \ref{thm:wellposedness}, one can easily see that there exists a constant $C \gs 0$, depending only on the parameters of the system and on $\Lambda$, such that 
	\begin{equation}
		\label{remainder}
		\norm{R_1^h}_{L^\infty(Q_T)}, \norm{R_2^h}_{L^\infty(Q_T)}, \norm{R_3^h}_{L^\infty(Q_T)} \le C.
	\end{equation} 
    In order to show \eqref{frechet:aim}, we now proceed with a priori estimates on the system \eqref{eq:phidiff}--\eqref{icdiff}. Indeed, we test \eqref{eq:phidiff} by $\psi$, \eqref{eq:sigmadiff} by $\theta$, \eqref{eq:pdiff} by $\zeta$ and sum them up to obtain:
    \begin{equation}
    \label{frechet:test1}
        \begin{split}
            & \mezzo \ddt \left( \norm{\phi}^2_H + \norm{\sigma}^2_H + \norm{p}^2_H \right) + \lambda \norm{\nabla \phi}^2_H + \eta \norm{\nabla \sigma}^2_H + D \norm{\nabla p}^2_H + \gamma_h \norm{\theta}^2_H + \gamma_p \norm{\zeta}^2_H \\
            & \quad = - (F^h, \psi)_H + (M^h, \psi)_H + (Q^h, \theta)_H + S_{ch} (\psi, \theta)_H + \alpha_{ch} (\psi, \zeta)_H. 
        \end{split}
    \end{equation}
    The last two terms on the right-hand side of \eqref{frechet:test1} can easily be treated by means of Cauchy-Schwarz and Young's inequalities, therefore we focus on estimating the first three. Indeed, by using again H\"older and Young's inequalities, together with \eqref{remainder}, the global boundedness of $(\phib, \sigmab, \pb)$ given by Theorem \ref{thm:wellposedness}, the regularity and local Lipschitz continuity of $F$, $m$ and $\hh$ and the Sobolev embedding $V \hookrightarrow \Lx4$, we infer that
    {\allowdisplaybreaks
    \begin{align*}
        (F^h, \psi)_H & = \int_\Omega F''(\phib) \psi^2 + R_1^h (\phi - \phib)^2 \psi \, \de x\le C \norm{\psi}^2_H + C \norm{\phi - \phib}^4_{\Lx4} \\
        & \le C \norm{\psi}^2_H + C \norm{\phi - \phib}^4_{V_0}, \\
        (M^h, \psi)_H & = \int_\Omega (m(\sigma) - m(\sigmab))(\hh'(\phi) - \hh'(\phib)) \psi \, \de x + \int_\Omega m(\sigmab)[\hh''(\phib) \psi + R_2^h (\phi - \phib)^2] \psi \, \de x \\
        & \quad + \int_\Omega \hh'(\phi) [m'(\sigmab) \theta + R_3^h (\sigma - \sigmab)^2] \psi \, \de x \\
        & \le C \norm{\psi}^2_H + C \norm{\theta}^2_H + C \norm{\phi-\phib}^4_{\Lx4} + C \norm{\sigma - \sigmab}^4_{\Lx4} \\
        & \le C \norm{\psi}^2_H + C \norm{\theta}^2_H + C \norm{\phi-\phib}^4_{V_0} + C \norm{\sigma - \sigmab}^4_V, \\
        (Q^h, \theta)_H & = \gamma_{ch} \int_\Omega (\sigma - \sigmab)(\phi - \phib)\theta + \sigmab \psi \theta + \phib \theta^2 \, \de x \\
        & \le C \norm{\theta}^2_H + C \norm{\psi}^2_H + C \norm{\sigma - \sigmab}^2_{\Lx4} \norm{\phi - \phib}^2_{\Lx4} \\
        & \le C \norm{\psi}^2_H + C \norm{\theta}^2_H + C \norm{\phi-\phib}^4_{V_0} + C \norm{\sigma - \sigmab}^4_V.
    \end{align*}
    }
    Then, by integrating on $(0,t)$ for any $t \in (0,T)$, we get that
    \begin{align*}
        & \mezzo \norm{\psi(t)}^2_H + \mezzo \norm{\theta(t)}^2_H + \mezzo \norm{\zeta(t)}^2_H + \lambda \int_0^t \norm{\nabla \psi}^2_H \, \de s + \eta \int_0^t \norm{\nabla \theta}^2_H \, \de s + D \int_0^t \norm{\nabla \zeta}^2_H \, \de s \\
        & \quad \le C \int_0^T \norm{\psi}^2_H \, \de s + C \int_0^T \norm{\theta}^2_H \, \de s + C \int_0^T \norm{\zeta}^2_H \, \de s 
                + C \int_0^T \norm{\phi-\phib}^4_{V_0} \, \de s + C \int_0^T \norm{\sigma - \sigmab}^4_V \, \de s.
    \end{align*}
    Therefore, by applying Gronwall's lemma and the continuous dependence estimate found in Proposition \ref{prop:strongcontdep}, we conclude that
    \begin{equation}
        \label{frechet:energyest}
        \begin{split}
        & \norm{\psi}^2_{\LT\infty H \cap \LT 2 {V_0}} + \norm{\theta}^2_{\LT\infty H \cap \LT 2 V} \\
        & \quad + \norm{\zeta}^2_{\LT\infty H \cap \LT 2 V} \le C \left( \norm{h}^4_{V_0} + \norm{k}^4_V + \norm{w}^4_V \right).
        \end{split}
    \end{equation}
    Moreover, by comparison in \eqref{eq:phidiff}, \eqref{eq:sigmadiff} and \eqref{eq:pdiff}, starting from \eqref{frechet:energyest}, we also  deduce that 
    \begin{equation}
        \label{frechet:timeest}
        \norm{\psi}^2_{\HT 1 {V_0^*}} + \norm{\theta}^2_{\HT 1 {V^*}} + \norm{\zeta}^2_{\HT 1 {V^*}} \le C \left( \norm{h}^4_{V_0} + \norm{k}^4_V + \norm{w}^4_V \right).
    \end{equation}
    Therefore, due to the standard embedding $H^1(0,T;\VV^*) \cap L^2(0,T;\VV) \hookrightarrow \Ccal^0([0,T]; \HH)$, estimates \eqref{frechet:energyest} and \eqref{frechet:timeest} imply that 
    \[ \norm{(\psi(T), \theta(T), \zeta(T))}^2_{\HH} \le C \norm{(h,k,w)}^4_{\VV}, \]
    which is exactly \eqref{frechet:aim} with $s = 4 \gs 2$. Then, the operator $\Rcal$ is Fr\'echet-differentiable and its derivative is fully characterised.

    Next, we want to prove that the Fr\'echet derivative is Lipschitz continuous. To do this, we show that, given two sets of initial data $(\phib_0^i, \sigmab_0^i, \pb_0^i) \in \Iad$, $i=1,2$, with corresponding solutions $(\phib_i, \sigmab_i, \pb_i)$, $i=1,2$, then for any $(h,k,w) \in \VV \cap \Lx\infty^3$ it holds that
    \begin{equation}
        \label{c1class:aim}
        \begin{split}
            & \norm{\D\Rcal((\phib_0^1, \sigmab_0^1, \pb_0^1))[(h,k,w)] - \D\Rcal((\phib_0^2, \sigmab_0^2, \pb_0^2))[(h,k,w)]}^2_{\HH} \\
            & \quad = \norm{(Y_1(T), Z_1(T), P_1(T)) - (Y_2(T), Z_2(T), P_2(T))}^2_{\HH} \\
            & \le C_0 \norm{(\phib_0^1, \sigmab_0^1, \pb_0^1) - (\phib_0^2, \sigmab_0^2, \pb_0^2)}^2_{\HH}.
        \end{split}
    \end{equation}
    To prove \eqref{c1class:aim}, we consider the system solved by the differences $Y = Y_1 - Y_2$, $Z = Z_1 - Z_2$ and $P = P_1 - P_2$ of the corresponding solutions of the linearised system \eqref{eq:philin}--\eqref{iclin}, which, up to adding and subtracting some terms, has the following form: 
    {\allowdisplaybreaks
    \begin{alignat}{2}
        & \partial_t Y - \lambda \Delta Y + F''(\phib_1) Y + (F''(\phib_1) - F''(\phib_2))Y_2 \nonumber\\
        & \quad = m(\sigmab_1)\hh''(\phib_1) Y + (m(\sigmab_1) - m(\sigmab_2))\hh''(\phib_1) Y_2 \nonumber\\
        & \qquad + m(\sigmab_2)(\hh''(\phib_1) - \hh''(\phib_2))Y_2 + m'(\sigmab_1)\hh'(\phib_1) Z \nonumber\\
        & \qquad + (m'(\sigmab_1) - m'(\sigmab_2))\hh'(\phib_1) Z_2 + m'(\sigmab_1) (\hh'(\phib_1) - \hh'(\phib_2)) Z_2 \qquad && \hbox{in $Q_T$,} \label{eq:phic1} \\
        & \partial_t Z - \eta \Delta Z + \gamma_h Z + \gamma_{ch} \sigmab_1 Y + \gamma_{ch} (\sigmab_1 - \sigmab_2) Y_2 \nonumber\\
        & \qquad + \gamma_{ch} \phib_1 Z + \gamma_{ch} (\phib_1 - \phib_2) Z_2 = S_{ch} Y && \hbox{in $Q_T$,} \label{eq:sigmac1} \\
        & \partial_t P - D \Delta P + \gamma_p P = \alpha_{ch} Y && \hbox{in $Q_T$,} \label{eq:pc1} \\
        & Y = 0, \quad \partial_{\n} Z = \partial_{\n} P = 0 && \hbox{on $\Sigma_T$,} \label{bcc1} \\
        & Y(0) = 0, \quad Z(0) = 0, \quad P(0) = 0 && \hbox{in $\Omega$.} \label{icc1}
    \end{alignat}
    }
    Then, we test \eqref{eq:phic1} by $Y$, \eqref{eq:sigmac1} by $Z$, \eqref{eq:pc1} by $P$ and sum them up to obtain:
    \begin{equation}
        \label{c1class:test}
        \begin{split}
            & \mezzo \ddt \left( \norm{Y}^2_H + \norm{Z}^2_H + \norm{P}^2_H \right)\\ 
            & \qquad + \lambda \norm{\nabla Y}^2_H + \eta \norm{\nabla Z}^2_H + D \norm{\nabla P}^2_H + \gamma_h \norm{Z}^2_H + \gamma_p \norm{P}^2_H \\
            & \quad = (F''(\phib_1) Y, Y)_H + ((F''(\phib_1) - F''(\phib_2))Y_2, Y)_H + (m(\sigmab_1)\hh''(\phib_1) Y, Y)_H \\
            & \qquad + ((m(\sigmab_1) - m(\sigmab_2))\hh''(\phib_1) Y_2, Y)_H + (m(\sigmab_2)(\hh''(\phib_1) - \hh''(\phib_2))Y_2, Y)_H \\
            & \qquad + (m'(\sigmab_1)\hh'(\phib_1) Z, Y) + ((m'(\sigmab_1) - m'(\sigmab_2))\hh'(\phib_1) Z_2, Y)_H \\
            & \qquad + (m'(\sigmab_1) (\hh'(\phib_1) - \hh'(\phib_2)) Z_2, Y)_H - \gamma_{ch} (\sigmab_1 Y, Z)_H - \gamma_{ch} ((\sigmab_1 - \sigmab_2) Y_2, Z)_H \\
            & \qquad - \gamma_{ch} (\phib_1 Z, Z)_H - \gamma_{ch} ((\phib_1 - \phib_2) Z_2, Z)_H + (S_{ch} Y, Z)_H + \alpha_{ch} (Y,P)_H.
        \end{split}
    \end{equation}
    Now, by using H\"older and Young's inequalities, together with the global boundedness of $(\phib_i, \sigmab_i, \pb_i)$ given by Theorem \ref{thm:wellposedness}, the regularity and local Lipschitz continuity of $F$, $m$ and $\hh$ and the Sobolev embedding $V \hookrightarrow \Lx4$, we can easily estimate all the terms on the right-hand side of \eqref{c1class:test} in the following way:
    \begingroup
    \allowdisplaybreaks
    \begin{align*}
        & (F''(\phib_1) Y, Y)_H + (m(\sigmab_1)\hh''(\phib_1) Y, Y)_H + (m'(\sigmab_1)\hh'(\phib_1) Z, Y) \\
        & \qquad - \gamma_{ch} (\sigmab_1 Y, Z)_H - \gamma_{ch} (\phib_1 Z, Z)_H + (S_{ch} Y, Z)_H + \alpha_{ch} (Y,P)_H \\
        & \quad \le C \norm{Y}^2_H + C \norm{Z}^2_H + C \norm{P}^2_H, \\
        & ((F''(\phib_1) - F''(\phib_2))Y_2, Y)_H + ((m(\sigmab_1) - m(\sigmab_2))\hh''(\phib_1) Y_2, Y)_H \\
        & \qquad + (m(\sigmab_2)(\hh''(\phib_1) - \hh''(\phib_2))Y_2, Y)_H + ((m'(\sigmab_1) - m'(\sigmab_2))\hh'(\phib_1) Z_2, Y)_H \\
        & \qquad + (m'(\sigmab_1) (\hh'(\phib_1) - \hh'(\phib_2)) Z_2, Y)_H - \gamma_{ch} ((\sigmab_1 - \sigmab_2) Y_2, Z)_H - \gamma_{ch} ((\phib_1 - \phib_2) Z_2, Z)_H \\
        & \quad \le C \norm{\phib_1 - \phib_2}_{\Lx4} \norm{Y_2}_H \norm{Y}_{\Lx4} + C \norm{\sigmab_1 - \sigmab_2}_{\Lx4} \norm{Y_2}_H \norm{Y}_{\Lx4} \\
        & \qquad + C \norm{\sigmab_1 - \sigmab_2}_{\Lx4} \norm{Z_2}_H \norm{Y}_{\Lx4} + C \norm{\phib_1 - \phib_2}_{\Lx4} \norm{Z_2}_H \norm{Y}_{\Lx4} \\
        & \qquad + C \norm{\sigmab_1 - \sigmab_2}_{\Lx4} \norm{Y_2}_H \norm{Z}_{\Lx4} + C \norm{\phib_1 - \phib_2}_{\Lx4} \norm{Z_2}_H \norm{Z}_{\Lx4} \\
        & \quad \le \frac{\lambda}{2} \norm{Y}^2_{V_0} + \frac{\eta}{2} \norm{Z}^2_V + C \underbrace{\left( \norm{Y_2}^2_H + \norm{Z_2}^2_H \right)}_{\in \, \Lt\infty} \norm{\phib_1 - \phib_2}^2_{V_0} + C \underbrace{\left( \norm{Y_2}^2_H + \norm{Z_2}^2_H \right)}_{\in \, \Lt\infty} \norm{\sigmab_1 - \sigmab_2}^2_V, 
    \end{align*}
    \endgroup
    where $\norm{Y_2}^2_H$ and $\norm{Z_2}^2_H$ are uniformly bounded in $\Lt\infty$ by Proposition \ref{prop:linearised}. Therefore, by also integrating on $(0,t)$, for any $t \in (0,T)$, starting from \eqref{c1class:test}, we deduce that
    \begin{align*}
        & \mezzo \left( \norm{Y(t)}^2_H + \norm{Z(t)}^2_H + \norm{P(t)}^2_H \right) + \frac{\lambda}{2} \int_0^t \norm{\nabla Y}^2_H \, \de s + \frac{\eta}{2} \int_0^t \norm{\nabla Z}^2_H \, \de s + \frac{D}{2} \int_0^t \norm{\nabla P}^2_H \, \de s \\
        & \quad \le C \int_0^T \norm{Y}^2_H + \norm{Z}^2_H + \norm{P}^2_H \, \de s + C \int_0^T \norm{\phib_1 - \phib_2}^2_{V_0} + \norm{\sigmab_1 - \sigmab_2}^2_V \, \de s. 
    \end{align*}
    Hence, by using Gronwall's lemma and the continuous dependence estimate \eqref{contdep:est}, we infer that 
    \begin{equation}
        \label{c1class:energyest}
        \begin{split}
        & \norm{Y}^2_{\LT\infty H \cap \LT 2 {V_0}} + \norm{Z}^2_{\LT\infty H \cap \LT 2 V} \\
        & \quad + \norm{P}^2_{\LT\infty H \cap \LT 2 V} \le C \norm{(\phib_0^1, \sigmab_0^1, \pb_0^1) - (\phib_0^2, \sigmab_0^2, \pb_0^2)}^2_{\HH}.
        \end{split}
    \end{equation}
    Moreover, by comparison in \eqref{eq:phic1}, \eqref{eq:sigmac1} and \eqref{eq:pc1}, starting from \eqref{c1class:energyest}, we also deduce that 
    \begin{equation}
        \label{c1class:timeest}
        \norm{Y}^2_{\HT 1 {V_0^*}} + \norm{Z}^2_{\HT 1 {V^*}} + \norm{P}^2_{\HT 1 {V^*}} \le C \norm{(\phib_0^1, \sigmab_0^1, \pb_0^1) - (\phib_0^2, \sigmab_0^2, \pb_0^2)}^2_{\HH}.
    \end{equation}
    Therefore, due to the standard embedding $H^1(0,T;\VV^*) \cap L^2(0,T;\VV) \hookrightarrow \Ccal([0,T]; \HH)$, estimates \eqref{c1class:energyest} and \eqref{c1class:timeest} allow us to say that 
    \begin{equation*}
        \norm{(Y(T), Z(T), P(T))}^2_{\HH} \le C \norm{(\phib_0^1, \sigmab_0^1, \pb_0^1) - (\phib_0^2, \sigmab_0^2, \pb_0^2)}^2_{\HH},
    \end{equation*}
    which immediately implies \eqref{c1class:aim}. Note that all this procedure is independent of the particular choice of $(h,k,w) \in \VV \cap \Lx\infty^3$, thus by taking the supremum in $(h,k,w)$ one recovers exactly \eqref{lipconst:frechet}. This concludes the proof of Theorem \ref{thm:frechet}.
\end{proof}

\begin{remark}
    Note that, by Proposition \ref{prop:linearised}, the solution to the linearised system \eqref{eq:philin}--\eqref{iclin} is well-defined for any $(h,k,w) \in \VV$, without additionally asking that $(h,k,w) \in \VV \cap \Lx\infty^3$. This means that, for any fixed $(\phib_0, \sigmab_0, \pb_0) \in \Iad$, the Fr\'echet-derivative $\D\Rcal(\phib_0, \sigmab_0, \pb_0)$ can be extended to a continuous linear operator on the whole $\VV$.
\end{remark}

\section{Analysis of the inverse problem}
\label{sec3}

We devote this section to the analysis of the inverse problem of reconstructing the initial data, given a measurement at the final time.
The first step would be to prove a backward uniqueness result for the system \eqref{eq:phi}--\eqref{ic}, namely to show that if $(\phi_1, \sigma_1, p_1)$ and $(\phi_1, \sigma_1, p_1)$ are two solutions such that $(\phi_1, \sigma_1, p_1)(T) = (\phi_2, \sigma_2, p_2)(T)$, then $(\phi_1, \sigma_1, p_1)(t) = (\phi_2, \sigma_2, p_2)(t)$ for any $t \in [0,T]$. In particular, this would imply that the map $\mathcal{R}$ is injective i.e. uniqueness for the inverse problem in $\Iad$.
However, following a logarithmic convexity approach, we prove a conditional stability result which consequently implies uniqueness as a byproduct. Namely, assuming the initial data to belong to $\Iad$, we prove an H\"older stability estimate for $\norm{(\phi_1(t), \sigma_1(t), p_1(t)) - (\phi_2(t), \sigma_2(t), p_2(t)}_{\HH}$ for any $t \gs 0$. To establish this, we exploit the global bounds on the solutions of \eqref{eq:phi}--\eqref{ic}, obtained in Theorem \ref{thm:wellposedness}. 
In particular, in the sequel we will denote by $M > 0$ the minimal constant such that 
\begin{equation}
    \label{bound:c0h}
    \norm{(\phi, \sigma, p)}_{\C 0 {\HH}} \le M,
\end{equation}
uniformly for $(\phi_0, \sigma_0, p_0) \in \Iad$. 
To prove our conditional stability result, we follow \cite[Theorem 3.1.3]{isakov}, which is actually a simpler version of the results contained in \cite{AN1967}. Indeed, we can prove:

\begin{proposition}
	\label{prop:condstab}
	Assume hypotheses \ref{ass:coeff}--\ref{ass:m}. Let $(\phi_1, \sigma_1, p_1)$ and $(\phi_2, \sigma_2, p_2)$ be two solutions of \eqref{eq:phi}--\eqref{ic} corresponding to two triples of initial data $(\phi_0^i, \sigma_0^i, p_0^i) \in \Iad$ for $i=1,2$.
	
	Then, there exists a constant $C_1 \gs 0$, depending only on the parameters of the system, such that for any $t \in [0,T]$ the following estimate holds:
    \begin{equation}
		\label{eq:condstab}
		\begin{split}
		& \norm{(\phi_1(t), \sigma_1(t), p_1(t)) - (\phi_2(t), \sigma_2(t), p_2(t))}_{\HH} \\
		& \quad \le C_1 M^{1-\lambda(t)} \norm{(\phi_1(T), \sigma_1(T), p_1(T)) - (\phi_2(T), \sigma_2(T), p_2(T))}_{\HH}^{\lambda(t)}.
		\end{split}
	\end{equation}
	where $\lambda = \lambda(t) \ge 0$ can be either 
	\begin{equation}
		\label{lambda:choice}
		\lambda = \lambda_1(t) = \frac{1 - e^{-\gamma t}}{1 - e^{- \gamma T}} \quad \text{or} \quad \lambda = \lambda_2(t) = \frac{e^{\gamma t} - 1}{e^{\gamma T} -1},
	\end{equation}
	for any $t \in [0,T]$, with $\gamma \gs 0$ depending only on the fixed parameters of the system.
\end{proposition} 

\begin{proof}
    We call $\phi = \phi_1 - \phi_2$, $\sigma = \sigma_1 - \sigma_2$, $p= p_1 - p_2$ and observe that they satisfy the system:
	\begin{alignat}{2}
		& \partial_t \phi - \lambda \Delta \phi = f_\phi
		&& \hbox{in $Q_T$,} \label{eq:phi2} \\
		& \partial_t \sigma - \eta \Delta \sigma = f_\sigma 
		&& \hbox{in $Q_T$,} \label{eq:sigma2} \\
		& \partial_t p - D \Delta p = f_p
		&& \hbox{in $Q_T$,} \label{eq:p2} \\
		& \phi = 0, \quad \partial_{\n} \sigma = \partial_{\n} p = 0 
		&& \hbox{in $\Sigma_T$,} \label{bc2} \\
		& \phi(0) = \phi_0^1 - \phi_0^2, \quad \sigma(0) = \sigma_0^1 - \sigma_0^2, \quad p(0) = p_0^1 - p_0^2
		\qquad && \hbox{in $\Omega$,} \label{ic2}
	\end{alignat}
	where
	\begin{align*}
		& f_\phi := - ( F'(\phi_1) - F'(\phi_2) ) + m(\sigma_1) (\hh'(\phi_1) - \hh'(\phi_2)) + (m(\sigma_1) - m(\sigma_2)) \hh'(\phi_2), \\
		& f_\sigma := - \gamma_h \sigma - \gamma_{ch} \sigma_1 (\phi_1 - \phi_2) - \gamma_{ch} (\sigma_1 - \sigma_2) \phi_2 + S_{ch} \phi, \\
		& f_p := - \gamma_p p + \alpha_{ch} \phi. 
	\end{align*}
	We start by estimating the terms on the right-hand side. Indeed, by using the fact that $\phi_i, \sigma_i$, $i=1,2$, are uniformly bounded in $\Lqt\infty$ and $F, \hh \in \mathcal{C}^2(\R)$, $m\in \mathcal{C}^1(\R)$ and thus locally Lipschitz, we can infer that
	\begin{align*}
		\norm{f_\phi}_H & \le \norm{F'(\phi_1) - F'(\phi_2)}_H + \norm{m(\sigma_1)}_\infty \norm{\hh'(\phi_1) - \hh'(\phi_2)}_H
  + \norm{\hh'(\phi_2)}_\infty \norm{m(\sigma_1) - m(\sigma_2)}_H \\ 
			& \le C \left( \norm{\phi}_H + \norm{\sigma}_H \right).
	\end{align*}
	Similarly, we can also deduce that
	\begin{align*}
		\norm{f_\sigma}_H & \le \gamma_h \norm{\sigma}_H + \gamma_{ch} \norm{\sigma_1}_\infty \norm{\phi_1 - \phi_2}_H + \gamma_{ch} \norm{\phi_2}_\infty \norm{\sigma_1 - \sigma_2}_H + S_{ch} \norm{\phi}_H \\
		& \le C \left( \norm{\sigma}_H + \norm{\phi}_H \right), 
	\end{align*}
	and that
	\[ \norm{f_p}_H \le C \left( \norm{p}_H + \norm{\phi}_H \right). \]
	Therefore, in the end, we showed that 
	\begin{equation}
		\label{backuniq:eq1}
		\norm{f_\phi}_H + \norm{f_\sigma}_H + \norm{f_p}_H \le C \left( \norm{\phi}_H + \norm{\sigma}_H + \norm{p}_H \right), \quad
  \text{for a.e. } t\in [0,T].
	\end{equation}
	To simplify notations, we now introduce a vector-valued formulation by defining: 
	\begin{align*}
		& \vpsi := (\phi, \sigma, p), \quad \nabla \vpsi := (\nabla \phi, \nabla \sigma, \nabla p), \\  & \Delta \vpsi := (\Delta \phi, \Delta \sigma, \Delta p), \quad f_{\vpsi} := (f_\phi, f_\sigma, f_p), \quad \vD := \diag (\lambda, \eta, D).
	\end{align*}
	Then, the previous estimate \eqref{backuniq:eq1} now becomes
	\begin{equation}
	\label{backuniq:fpsi}
		\norm{f_{\vpsi}}_{\HH} \le C \norm{\vpsi}_{\HH}.
	\end{equation}
    Thus, we have shown that the system solved by $(\phi, \sigma, p)$ has the form
	\[ \partial_t \vpsi - \vD \Delta \vpsi = f_{\vpsi}, \hbox{ with $\norm{f_{\vpsi}}_{\HH} \le C \norm{\vpsi}_{\HH}$}, \]
	where $C$ depends only on the parameters of the system and on $\norm{\phi_0^i}_\infty, \norm{\sigma_0^i}_\infty$, $i=1,2$.
	Next, recall that $\vpsi \in \XX_0 \times \XX \times \XX$, therefore, by respectively testing \eqref{eq:phi2} by $\phi$ and $-\Delta \phi$, \eqref{eq:sigma2} by $\sigma$ and $-\Delta \sigma$, \eqref{eq:p2} by $p$ and $-\Delta p$ and summing them up, we obtain the identities:
	\begin{align}
		& \mezzo \ddt \norm{\vpsi}^2_{\HH} + \norm{\vD \nabla \vpsi}^2_{\HH} = (f_{\vpsi}, \vpsi)_{\HH}, \label{backuniq:test1} \\ 
		& \mezzo \ddt \norm{\nabla \vpsi}^2_{\HH} + \norm{\vD \Delta \vpsi}^2_{\HH} = (f_{\vpsi}, - \Delta \vpsi)_{\HH}. \label{backuniq:test2}
	\end{align}
    Now, we are exactly in the situation to apply \cite[Theorem 3.1.3]{isakov}, with $\alpha = C$, to obtain the desired estimate. In particular, observe that, by going through the proof of the cited Theorem and especially the one of \cite[Lemma 3.1.4]{isakov}, the argument can be repeated even in our regularity setting with $\vpsi \in \XX_0 \times \XX \times \XX$. Indeed, the crucial point is the validity of the identities \eqref{backuniq:test1} and \eqref{backuniq:test2}. For the sake of completeness, we recall here below the fundamental steps of the proof. 

    Let 
    \[ q := \norm{\vpsi}^2_{\HH}, \quad g := \frac{2(f_{\vpsi}, \vpsi)_{\HH}}{q}, \quad l := \log q - \int_0^t g \, \de s. \]
    Then, the first step is to show, through explicit calculations, that there exists a constant $c \ge 0$ such that $l$ satisfies the differential inequality 
    \begin{equation}
        \label{l:diffineq}
        \partial^2_t l + c \abs{\partial_t l} + c \ge 0 \quad \text{in $(0,T)$}.
    \end{equation}
    This can be done exactly as in \cite[Lemma 3.1.4]{isakov}, by relying on the identities \eqref{backuniq:test1} and \eqref{backuniq:test2}. Note that these calculations are actually easier in our case, since our operator $A = - \vD \Delta$ has no skew-symmetric part. Then, exactly as in \cite[Lemma 3.1.5]{isakov}, by exploiting maximum, minimum and comparison principles, we can see that $l(t) \le L(t)$ for any $t \in (0,T)$, where $L$ is the solution to the following boundary value problem:
    \[ \begin{cases}
        \partial_t^2 L + c_\tau \partial_t L + c = 0, \\
        L(0) = l(0), \quad L(T) = l(T),
    \end{cases} \quad \text{where } c_\tau = \begin{cases}
        c \quad \text{on } (0,\tau), \\
        - c \quad \text{on } (\tau, T),
    \end{cases} \]
    for some $\tau \in (0,T)$, depending on the sign of $\partial_t L$. Let us first consider the problem
    \[ \begin{cases}
        \partial_t^2 v + c_\tau \partial_t v = 0, \\
        v(0) = l(0), \quad v(T) = l(T).
    \end{cases} \]
    Observe that, since $\partial_t v$ satisfies a homogeneous linear differential equation of first order, it does not change its sign, which depends only on the values $v(0)$ and $v(T)$. Then, by the maximum principle we can say that $v \le V$ on $[0,T]$, with $V$ solution to
    \[ \begin{cases}
        \partial_t^2 V + c \partial_t V = 0, \\
        V(0) = l(0), \quad V(T) = l(T).
    \end{cases} \]
    Computing directly the solution $V$ we find that
    \[ v(t) \le V(t) = l(0) \frac{e^{-\gamma t} - e^{-\gamma T}}{1 - e^{-\gamma T}} + l(T) \frac{1 - e^{-\gamma t}}{1 - e^{-\gamma T}}, \]
    for any $t \in [0,T]$, with $\gamma$ being either $c$ or $-c$, depending on whether $l(0) \le l(T)$ or $l(0) > l(T)$. Then, one can also find an explicit solution $w$ of the non-homogeneous equation with zero boundary data and show that 
    \[  w(t) \le \frac{2 e^{cT}}{c} \quad \text{for any } t \in (0,T). \]
    We refer to \cite[Lemma 3.1.5]{isakov} for all these explicit calculations. Hence, we find that 
    \[ l(t) \le L(t) = v(t) + w(t) \le l(0) \frac{e^{-\gamma t} - e^{-\gamma T}}{1 - e^{-\gamma T}} + l(T) \frac{1 - e^{-\gamma t}}{1 - e^{-\gamma T}} + \frac{2 e^{cT}}{c},  \]
    for any $t \in [0,T]$, with $\gamma = c$ or $\gamma = -c$. Therefore, by using the definition of $l$ and the fact that $\int_0^t g \, \de s \le 2CT$, we deduce that 
    \[ \log \norm{\vpsi(t)} \le \log \norm{\vpsi(0)} \frac{e^{-\gamma t} - e^{-\gamma T}}{1 - e^{-\gamma T}} + \log \norm{\vpsi(T)} \frac{1 - e^{-\gamma t}}{1 - e^{-\gamma T}} + \frac{2 e^{cT}}{c} + 2CT. \]
    Finally, by exponentiating and using $\norm{\vpsi(0)}_{\HH} \le M$, we reach the thesis with a constant $C_1 > 0$ depending only on $c$, $C$ and $T$.
\end{proof}

\begin{remark}
	In the proof of Proposition \ref{prop:condstab}, for simplicity of exposition, we used the abstract formulation $\partial_t \vpsi - \vD \Delta \vpsi = f_{\vpsi}$, $\vpsi(0) = \vpsi_0$ of the system \eqref{eq:phi2}--\eqref{ic2}. We just want to comment on the fact that the Laplacian operator is not exactly the same for all three components, since for $\phi$ it is complemented with homogeneous Dirichlet boundary conditions, while for $\sigma$ and $p$ with homogeneous Neumann. However, with a slight abuse of notation, we will stick to the already-used expression.
\end{remark}

\begin{remark}
	Observe that, from \cite[Lemma 3.1.5]{isakov}, the choice of $\lambda_1$ or $\lambda_2$ depends on whether $\log \norm{\vpsi(0)}^2_{\HH}$ is smaller or bigger than $\log \norm{\vpsi(T)}^2_{\HH} - \int_0^T (f_{\vpsi},\vpsi)/\norm{\vpsi}_{\HH}^2 \, \de t$, which in general cannot be verified a priori. Therefore, we need to take into account both versions of $\lambda(t)$. However, in both cases $\lambda_i(0) = 0$, $i=1,2$, so \eqref{eq:condstab} does not give a H\"older stability estimate also for $t=0$. Nevertheless, we still get a backward uniqueness result as in the following Corollary.
\end{remark}

\begin{corollary}
	\label{prop:backuniq}
	Assume hypotheses \ref{ass:coeff}--\ref{ass:m}. Let $(\phi_1, \sigma_1, p_1)$ and $(\phi_2, \sigma_2, p_2)$ be two solutions of \eqref{eq:phi}--\eqref{ic} corresponding to two triples of initial data $(\phi_0^i, \sigma_0^i, p_0^i) \in \Iad$ for $i=1,2$.
	
	If $(\phi_1, \sigma_1, p_1)(T) = (\phi_2, \sigma_2, p_2)(T)$, then $(\phi_1, \sigma_1, p_1)(t) = (\phi_2, \sigma_2, p_2)(t)$ for any $t \in [0,T]$. In particular, $(\phi_0^1, \sigma_0^1, p_0^1) = (\phi_0^2, \sigma_0^2, p_0^2)$ in $\VV$.
\end{corollary}

\begin{proof}
    Observe that, if $(\phi_1, \sigma_1, p_1)(T) = (\phi_2, \sigma_2, p_2)(T)$, then estimate \eqref{eq:condstab} gives that 
    \[ 
        \norm{(\phi_1(t), \sigma_1(t), p_1(t)) - (\phi_2(t), \sigma_2(t), p_2(t))}_{\HH} = 0
    \]
    for any $t \in (0,T]$.
    The result now easily follows also for $t=0$, by exploiting the continuity of the function $t \to \norm{(\phi_1(t), \sigma_1(t), p_1(t)) - (\phi_2(t), \sigma_2(t), p_2(t))}_{\HH}$, which is guaranteed by the $\C 0 {\HH}$-regularity of the solutions. 
\end{proof}

For future use, we observe that if we define (cf.~\eqref{bound:c0h} for the definition of $M$)
\begin{equation}
	\label{eq:epsilon}
	\eps := \frac{\norm{(\phi_1(T), \sigma_1(T), p_1(T)) - (\phi_2(T), \sigma_2(T), p_2(T))}_{\HH}}{M} = \frac{\norm{\vpsi(T)}_{\HH}}{M},
\end{equation}
then the stability estimate \eqref{eq:condstab} becomes 
\[ \norm{(\phi_1(t), \sigma_1(t), p_1(t)) - (\phi_2(t), \sigma_2(t), p_2(t))}_{\HH} \le C_1 M \eps^{\lambda(t)}, \]
for any $t \in (0,T)$, where $\lambda(t)$ is either $\lambda_1(t)$ or $\lambda_2(t)$.
To obtain a stability estimate even for the reconstruction of the initial data, we have to impose a further a priori bound. 
Indeed, if, additionally to $\triplein \in \Iad$, we also assume that
\[ \norm{(\phi_0, \sigma_0, p_0)}_{\VV} \le \Cb, \]
for some $\Cb > 0$, then by Theorem \ref{thm:wellposedness} we have the following uniform bound on the solution of \eqref{eq:phi}--\eqref{ic}:
\begin{equation}
    \label{bound:h1h}
    \norm{(\phi, \sigma, p)}_{\HT1{\HH}} \le M_1,
\end{equation}
where $M_1$ depends only on the fixed parameters of the system and on $\Cb$, and can be chosen as the minimal one. 
This further bound enables us to prove a logarithmic-type stability result for the reconstructed initial data.

\begin{theorem}
	\label{thm:inizstab}
	Assume hypotheses \ref{ass:coeff}--\ref{ass:m}. Let $(\phi_1, \sigma_1, p_1)$ and $(\phi_2, \sigma_2, p_2)$ be two solutions of \eqref{eq:phi}--\eqref{ic} corresponding to two triples of initial data $(\phi_0^i, \sigma_0^i, p_0^i) \in \Iad$ for $i=1,2$.
	
	Assume that for $i=1,2$ we also have the following a priori bound:
	\begin{equation}
		\label{eq:ass_aprioribound}
		\norm{(\phi_0^i, \sigma_0^i, p_0^i)}_{\VV} \le \Cb,
	\end{equation}
    and let $M \gs 0$, $C_1>0$, $M_1 \gs 0$ be as in \eqref{bound:c0h}, \eqref{eq:condstab}, \eqref{bound:h1h}.
	Let $\eps$ be defined as in \eqref{eq:epsilon}
	and assume that 
	\[ \eps \le \exp \left\{ - \left( \min \left\{ 1, \frac{4\sqrt{3} M C_1^{3/2}}{9M_1} \right\} \right)^{-1} \right\}. \]
	Then, there exists a constant $C_2 \gs 0$ such that 
	\begin{equation}
		\label{eq:inizstab}
		\norm{ (\phi_0^1, \sigma_0^1, p_0^1) - (\phi_0^2, \sigma_0^2, p_0^2) }^2_{\HH} \le \frac{C_2}{\sqrt{\abs{\log \eps}}}, 
	\end{equation}
	where we can quantify the constant as 
	\[ C_2 = \frac{2M_1 M C_1^{1/2}}{\beta^{1/2}} + \frac{3M_1^2}{4\beta C_1}, \quad \text{with } \beta = \frac{\gamma}{e^{\gamma T}-1} \gs 0, \]
 and where $\gamma \gs 0$ is the constant appearing in \eqref{lambda:choice}.
\end{theorem}

\begin{proof}
	For all this argument, we still use the compact notation with $\vpsi$ introduced in the proof of Proposition \ref{prop:condstab}. Note that under our regularity assumptions, by standard results in Lebesgue-Bochner theory, the function $t \to \norm{\vpsi(t)}^2_{\HH}$ is absolutely continuous. Then, we can apply the fundamental theorem of calculus to say that for any $t \in [0,T]$
	\begin{align*}
		\norm{\vpsi(t)}^2_{\HH} & = \norm{\vpsi(0)}^2_{\HH} + \int_0^t \frac{\de}{\de s} \norm{\vpsi(s)}^2_{\HH} \, \de s \\
		& = \norm{\vpsi(0)}^2_{\HH} + \int_0^t 2 ( \partial_s \vpsi(s), \vpsi(s) )_{\HH} \, \de s.
	\end{align*}
	Then, by using Cauchy-Schwarz, the stability inequality \eqref{eq:condstab} and the bounds \eqref{bound:c0h} and \eqref{bound:h1h}, we deduce that 
	\begin{align*}
		\norm{\vpsi(0)}^2_{\HH} & \le \norm{\vpsi(t)}^2_{\HH} + 2 \left( \int_0^t \norm{\partial_s \vpsi(s)}^2_{\HH} \, \de s \right)^{1/2} \left( \int_0^t \norm{\vpsi(s)}^2_{\HH} \, \de s \right)^{1/2} \\
		& \le C_1^2 M^2 \eps^{2\lambda(t)} + 2 M_1 \left( \int_0^t M \cdot C_1 M \eps^{\lambda(s)} \, \de s  \right)^{1/2},
	\end{align*}
	where $\lambda$ can be either $\lambda_1$ or $\lambda_2$. 
	Now, by working on the explicit expressions of $\lambda_i(t)$, $i=1,2$, one can easily see that $\lambda_i(0) = 0$, $\lambda_i(T) = 1$, $\lambda_i$ is strictly increasing for both $i=1,2$, $\lambda_1$ is concave and $\lambda_2$ is convex. 
	So one has the situation depicted in Figure \ref{fig:lambdas}. 
	Then, the tangent line at $0$ to $\lambda_2$, which has the expression
	\[ \tilde{\lambda}(t) = \beta t, \quad \text{with } \beta = \frac{\gamma}{e^{\gamma T}-1}, \]
	is below both $\lambda_1$ and $\lambda_2$. Therefore, since $0 \ls \eps \ls 1$, we can infer that 
	\[ \eps^{\lambda_i(t)} \le \eps^{\beta t} \quad \text{for both } i=1,2. \]
	Hence, we can further estimate
	\begin{align*}
		\norm{\vpsi(0)}^2_{\HH} & \le C_1^2 M^2 \eps^{2\beta t} + 2 M_1 M C^{1/2} \left( \int_0^t \eps^{\beta s} \, \de s  \right)^{1/2} \\
		& = C_1^2 M^2 \eps^{2\beta t} + 2 M_1 M C^{1/2} \left( \frac{\eps^{\beta t}}{\beta \log \eps} - \frac{1}{\beta \log \eps} \right)^{1/2} \\
		& = C_1^2 M^2 \eps^{2\beta t} + \frac{2 M_1 M C^{1/2}}{\beta^{1/2} \sqrt{\abs{\log \eps}}} (1 - \eps^{\beta t})^{1/2} := g(t).
	\end{align*}
	Now, note that this estimate holds for any $t \in [0,T]$, so we can make it sharper and independent of $t$ by minimising the function $g(t)$. Then, we compute its derivative and, recalling that $\log \eps \ls 0$, we obtain that 
	\[ g'(t) = - C_1^2 M^2 2 \beta \abs{\log \eps} \cdot \eps^{2\beta t} + M_1 M C_1^{1/2} \beta^{1/2} \radice \frac{\eps^{\beta t}}{(1-\eps^{\beta t})^{1/2}}. \]
	If we call $x = \eps^{\beta t}$, $0 \ls x \ls 1$, we can find the critical points by solving the equation
	\[ - C_1^2 M^2 2 \beta \abs{\log \eps} \cdot x^2 + M_1 M C_1^{1/2} \beta^{1/2} \radice \frac{x}{(1-x)^{1/2}} = 0, \]
	which, after some manipulations, is equivalent to
	\begin{equation}
		\label{eq:xsolve}
		x \sqrt{1-x} = \frac{M_1}{2 \beta^{1/2} C_1^{3/2} M \radice}.
	\end{equation}
	By analysing the function $f(x)=x \sqrt{1-x}$, we see that its graph is like the one in Figure \ref{fig:root_function}, with maximum value at $\frac{2\sqrt{3}}{9}$. Then, equation \eqref{eq:xsolve} has at least a solution if 
	\[ \frac{M_1}{2 \beta^{1/2} C_1^{3/2} M \radice} \le  \frac{2\sqrt{3}}{9}, \]
	which gives part of the smallness condition in the statement of the Theorem. 
	Moreover, one can easily see that the solution corresponding to the minimum value of $g(t)$ is the smallest one, let us call it $\bar{x}$. 
	Then, we can see that the increasing branch of the function $f$ lies between the two lines $y=x$ and $y=\frac{\sqrt{3}}{3}x$, therefore we can estimate the value of $\bar{x}$ with the solutions corresponding to the two lines, namely
	\begin{equation}
		\label{eq:xestimate}
		\frac{M_1}{2 \beta^{1/2} C_1^{3/2} M \radice} \le \bar{x} \le \frac{\sqrt{3}M_1}{2 \beta^{1/2} C_1^{3/2} M \radice}.
	\end{equation}
	Finally, we can use this piece of information to close our estimate:
	\begin{align*}
		\norm{\vpsi(0)}^2_{\HH} & \le \min_{t \in [0,T]} g(t) = C_1^2 M^2 \bar{x}^2 + \frac{2 M_1 M C^{1/2}}{\beta^{1/2} \sqrt{\abs{\log \eps}}} (1 - \bar{x})^{1/2} \\
		& \le C_1^2M^2 \frac{3M_1^2}{4\beta C_1^3 M^2 \abs{\log \eps}} + \frac{2 M_1 M C^{1/2}}{\beta^{1/2} \radice} \underbrace{\left( 1 - \frac{M_1}{2 \beta^{1/2} C_1^{3/2} M \radice} \right)^{1/2}}_{\le 1} \\
		& \le \frac{1}{\radice} \left( \frac{2 M_1 M C^{1/2}}{\beta^{1/2}} + \frac{3M_1^2}{4 \beta C_1 \radice} \right) \\
		& \le \frac{1}{\radice} \left( \frac{2 M_1 M C^{1/2}}{\beta^{1/2}} + \frac{3M_1^2}{4 \beta C_1} \right), 
	\end{align*}
	if we also assume that $\frac{1}{\radice} \le 1$, which gives the other part of the smallness assumption. 
	This concludes the proof of Theorem \ref{thm:inizstab}.
\end{proof}

\begin{figure}[t]
    \centering
    \begin{minipage}{0.5\textwidth}
        \centering
        \includegraphics[scale=0.8]{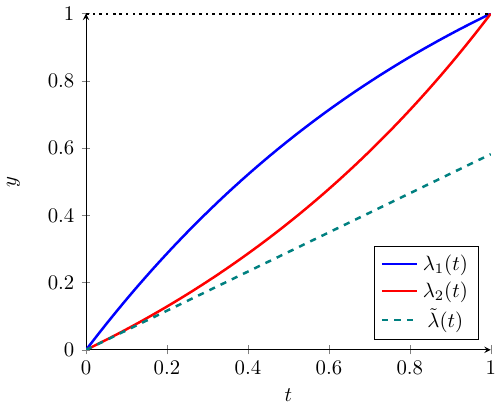}
        \captionof{figure}{Graphs of $\lambda_1(t)$, $\lambda_2(t)$ and $\tilde{\lambda}(t)$,\\ where we set  $T=1$ for convenience.}
        \label{fig:lambdas}
    \end{minipage}%
    \begin{minipage}{0.5\textwidth}
        \centering
        \includegraphics[scale=0.8]{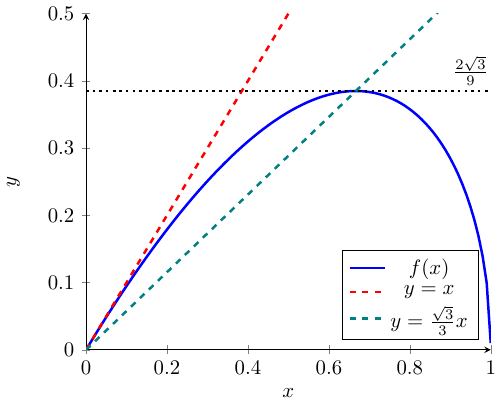}
        \captionof{figure}{Graph of the function $f(x) = x \sqrt{1-x}$, \\ along with the two linear bounds.}
        \label{fig:root_function}
    \end{minipage}
\end{figure}

\begin{remark}
	Observe that this type of logarithmic estimate can be impractical in applications, as the error actually gets small only if the final data are very close to each other.  
    Moreover, it can be easily seen that the stability constant $C_2$ deteriorates exponentially with the final time $T$, making this stability estimate even more unreliable. 
    Nevertheless, this is expected since we are dealing with a backward problem for a parabolic system, which is known to be severely ill-posed as $T$ gets larger.
\end{remark}

We now want to improve this stability estimate to one of Lipschitz type, under the additional a priori assumption that the initial data to be reconstructed actually lives in a finite-dimensional subspace of $\VV$, for instance, one of the discrete spaces used in numerical approximations. 
To do this, we recall the following general result, which was proved in \cite[Proposition 5]{BV2006}. 
We note that the same result was successfully used to get Lipschitz stability estimates for other inverse ill-posed problems, like in \cite{AV2005, ABFV2022, BFV2008}.
\begin{lemma}
	\label{lem:vessella}
	Let $R$ and $r_0$ be positive numbers. Let $\Lambda$ be an open subset of $\R^n$, for some $n \in \N$, and let $K$ be a subset of $\Lambda$. Assume that 
	\begin{equation}
		\label{eq:ass_finitedim}
		K \subseteq B_R(0) \quad \text{and} \quad B_{r_0}(p) \subseteq \Lambda, \text{ for any } p\in K.
	\end{equation}
	Let $\mathcal{B}$ be a Banach space and $G: \Lambda \to \mathcal{B}$ an operator. Assume that $G$ satisfies the following conditions:
	\begin{itemize}
		\item[$(i)$] $G_{\mid K}$ is injective,
		\item[$(ii)$] $(G_{\mid K})^{-1}: G(K) \to K$ is uniformly continuous with modulus of continuity $\omega_0(\cdot)$,
		\item[$(iii)$] $G$ is Fr\'echet-differentiable with derivative $\D G$,
		\item[$(iv)$] $\D G: \Lambda \to \mathcal{L}(\Lambda, \mathcal{B})$ is uniformly continuous with modulus of continuity $\omega_1(\cdot)$.
		\item[$(v)$] there exists $m_0 \gs 0$ such that 
							\[ \inf_{x \in K, \abs{y}=1} \norm{\D G(x)[y]}_{\mathcal{B}} \ge m_0. \]
	\end{itemize}
	
	Then, for any $x_1, x_2 \in K$
	\begin{equation}
		\abs{x_1 - x_2} \le C \norm{G(x_1) - G(x_2)}_{\mathcal{B}}, 
	\end{equation}
	where $C = \max\left\{ \frac{2R}{(\omega_0^{-1})_*(\delta_1)}, \frac{2}{m_0} \right\}$, $\delta_1 = \mezzo \min \{ \delta_0, r_0 \}$ and $\delta_0 = (\omega_1^{-1})_* \left( \frac{m_0}{2} \right)$. Here, for a modulus of continuity $\omega$, we denote
	\[ (\omega^{-1})_*(y) := \inf \{ x  \in [0, +\infty) \mid \omega(x) \ge y \}. \]
\end{lemma}


Before using Lemma \ref{lem:vessella} in our case, we still need to study more in detail the Fr\'echet-derivative of our operator $\mathcal{R}: \VV \to \HH$, which associates to any initial data $(\phi_0, \sigma_0, p_0)$ the solution $(\phi(T), \sigma(T), p(T))$ at the final time. 
In Theorem \ref{thm:frechet} we showed that $\D\mathcal{R}: \VV \cap \Lx\infty^3 \to \mathcal{L}(\VV \cap \Lx\infty^3, \HH)$ is such that, for any $(h,k,w) \in \VV \cap \Lx\infty^3$, $\D\mathcal{R}(\phi_0, \sigma_0, p_0)[h,k,w] = (Y(T), Z(T), P(T))$, which is the solution of the linearised system \eqref{eq:philin}--\eqref{iclin} evaluated at the final time. 
Regarding this linearised inverse problem, we can prove a similar logarithmic stability estimate, since it essentially has the same formal structure as \eqref{eq:phi2}--\eqref{ic2}. 

\begin{theorem}
	\label{thm:inizstab_lin}
	Assume hypotheses \ref{ass:coeff}--\ref{ass:m}. Let $(\phib, \sigmab, \pb)$ be a solution of \eqref{eq:phi}--\eqref{ic} corresponding to a triple of initial data $(\phi_0, \sigma_0, p_0) \in \Iad$. Moreover, let $(Y,Z,P)$ be a solution of \eqref{eq:philin}--\eqref{iclin} corresponding to $(h,k,w) \in \VV$. 
	Assume also that 
    \[ \norm{(h,k,w)}_{\VV} \le \Qb, \]
    for some $\Qb \gs 0$ and call $L = L(\Qb) \gs 0$, $L_1 = L_1 (\Qb) \gs 0$ the minimal constants such that by Proposition \emph{\ref{prop:linearised}}
	\begin{equation}\label{def:LL1} \norm{(Y,Z,P)}_{\C0{\HH}} \le L, \quad \norm{(Y,Z;P)}_{\HT1{\HH}} \le L_1. \end{equation}
	Then, there exists a constant $Q_1 \gs 0$, depending only on the parameters of the system, such that we have the following conditional stability estimate for any $t \in [0,T]$
	\[ \norm{(Y(t), Z(t), P(t))}_{\HH} \le Q_1 L^{1-\lambda} \norm{(Y(T), Z(T), P(T))}^\lambda,  \]
	where $\lambda$ is either $\lambda_1(t)$ or $\lambda_2(t)$.
	
	Moreover, if we additionally assume that 
	\[ \eps := \frac{\norm{(Y(T), Z(T), P(T))}_{\HH}}{L} \le \exp \left\{ - \left( \min \left\{ 1, \frac{4\sqrt{3} L Q_1^{3/2}}{9L_1} \right\} \right)^{-1} \right\}, \]
	then there exists a constant $Q_2 \gs 0$ such that 
	\begin{equation}
		\label{eq:inizstab_lin}
		\norm{ (h, k, w) }^2_{\HH} \le \frac{Q_2}{\sqrt{\abs{\log \eps}}}, 
	\end{equation}
	where we can quantify the constant as 
	\begin{equation}\label{def:Q2} Q_2 = \frac{2L_1 L C_1^{1/2}}{\beta^{1/2}} + \frac{3L_1^2}{4\beta Q_1}. \end{equation}
\end{theorem}

\begin{remark}
	Since the system is linear, it is clear that \eqref{eq:inizstab_lin} gives a stability estimate for the linearised inverse problem of reconstructing the initial data $(h,k,w)$ from the final measurements $(Y(T), Z(T), P(T))$; indeed, the difference of two solutions still satisfies the same linear system.
\end{remark}

\begin{proof}
	Observe that we can rewrite the linearised system \eqref{eq:philin}--\eqref{iclin} in the abstract form:
	\[ \partial_t \vy - \vD \Delta \vy = f_{\vy}, \quad \vy(0)  = \vh,  \]
	where $\vy = (Y,Z,P)^\top$, $\vh = (h,k,w)^\top$ and
	\[ f_{\vy} = \begin{pmatrix}
		- F''(\phib) Y + m(\sigmab) h''(\phib) Y + m'(\sigmab) h'(\phib) Z \\
		- \gamma_h Z - \gamma_{ch} \sigmab Y - \gamma_{ch} \phib Z + S_{ch} Y \\
		- \gamma_p P + \alpha_{ch} Y
	\end{pmatrix}. \] 
	We recall that the Laplacian operator actually depends on the different boundary conditions.
	By using the fact that $\phib, \sigmab \in \Lqt\infty$ by Theorem \ref{thm:wellposedness}, it is easy to see that 
	\[ \norm{f_{\vy}}_{\HH} \le C \norm{\vy}_{\HH}. \]
	Therefore, we are exactly in the same situation to apply \cite[Theorem 3.1.3]{isakov} and then follow the same reasoning of the proof of Theorem \ref{thm:inizstab}. This concludes the proof. 
\end{proof}

We are now in the position to apply Lemma \ref{lem:vessella} and establish a Lipschitz stability estimate, when starting from a finite-dimensional subspace.

\begin{theorem}
	\label{thm:lipstab}
	Assume hypotheses \ref{ass:coeff}--\ref{ass:m}. Let $\Lambda$ be a finite-dimensional subspace of $\VV$ and $K \subseteq \Lambda$ such that 
	\[ K = \{ \vec{u} \in \Lambda \mid \norm{\vec{u}}_{\VV} \le \Cb \}, \]
	for some constant $\Cb \gs 0$ to be decided a priori. Assume further that 
	\[ (\phi_0, \sigma_0, p_0) \in K \cap \Iad. \]
	Then, there exists a constant $C_s \gs 0$ such that for any choice of $(\phi_0^1, \sigma_0^1, p_0^1)$ and $(\phi_0^2, \sigma_0^2, p_0^2)$ in $K \cap \Iad$ the following stability estimate holds
	\begin{equation}
		\label{eq:lipstab}
		\norm{ (\phi_0^1, \sigma_0^1, p_0^1) - (\phi_0^2, \sigma_0^2, p_0^2) }_{\HH} \le C_s \norm{ (\phi_1(T), \sigma_1(T), p_1(T)) - (\phi_2(T), \sigma_2(T), p_2(T)) }_{\HH},
	\end{equation}
	where the constant $C_s$ can be quantified as 
	\[ C_s = \max \left\{ \frac{2 \Cb}{M} e^{ \frac{16 C_0^2 C_2}{m_0}}, \frac{2}{m_0} \right\}, \text{ with } m_0 = \frac{L}{C_{\Lambda}} e^{-Q^2_2} \text{ and } C_{\Lambda} = \sup_{\vh \in \Lambda \setminus \{0\}} \frac{\norm{\vh}_{\VV}}{\norm{\vh}_{\HH}}, \]
	where $C_0$ is the Lipschitz constant of the Fr\'echet-derivative, given by \eqref{lipconst:frechet}, and $M$, $L$, and $Q_2$ are defined in \eqref{bound:c0h}, \eqref{def:LL1}, \eqref{def:Q2}.
\end{theorem}

\begin{proof}
	We want to apply Lemma \ref{lem:vessella}, so we have to verify its hypotheses. 
    Indeed, we set $X = \VV$ and $\Lambda$ as above. Being a finite-dimensional subspace of $\VV$ means that $\Lambda = \R^n$ in the language of Lemma \ref{lem:vessella}, therefore we can immediately see that $K$ satisfies hypothesis \eqref{eq:ass_finitedim}.
	We consider the operator $\Rcal: \Lambda \subseteq \VV \to \HH$ defined as $\Rcal((\phi_0, \sigma_0, p_0)) = (\phi(T), \sigma(T), p(T))$. 
	Hypothesis $(i)$ on the injectivity of $\Rcal_{\mid K}$ follows directly from Corollary \ref{prop:backuniq}. 
	Also hypothesis $(ii)$ follows immediately from Theorem \ref{thm:inizstab}, with modulus of continuity $\omega_0(r) = C_2/\sqrt{\abs{\log(r/M)}}$. 
	Then, hypotheses $(iii)$ and $(iv)$ on Fr\'echet-differentiability follow from Theorem \ref{thm:frechet}, with modulus of continuity $\omega_1(r) = C_0 r$. Finally, we have to check condition $(v)$. 
	To do this, notice that for any $(\phi_0, \sigma_0, p_0) \in K$ and any $(h,k,w) \in \Lambda$, we have that $\D\Rcal (\phi_0, \sigma_0, p_0) [h,k,w] = (Y(T), Z(T), P(T))$, where $(Y,Z,T)$ is the solution to \eqref{eq:philin}--\eqref{iclin}. 
	Now, by Theorem \ref{thm:inizstab_lin}, we can use \eqref{eq:inizstab_lin} to say that, if $\norm{(h,k,w)}_{\VV} \le \Qb$, it holds that
	\[ \norm{(h,k,w)}^2_{\HH} \le \frac{Q_2}{\sqrt{\abs*{ \log \frac{\norm{(Y(T), Z(T), P(T))}_{\HH}}{L} }}}, \]
	which implies that 
	\[ \norm{(Y(T), Z(T), P(T))}_{\HH} \ge L e^{- \frac{Q_2^2}{\norm{(h,k,w)}^4_{\HH}}}. \]
	Hence, we can verify condition $(v)$ as
	\begin{align*}
		& \inf_{\substack{(\phi_0, \sigma_0, p_0) \in K \\ (h,k,w) \in \Lambda \setminus \{0\}}} \frac{\norm{(Y(T), Z(T), P(T))}_{\HH}}{\norm{(h,k,w)}_{\VV}} 
		\ge \inf_{\substack{(\phi_0, \sigma_0, p_0) \in K \\ (h,k,w) \in \Lambda \setminus \{0\}}} \frac{\norm{(Y(T), Z(T), P(T))}_{\HH}}{C_{\Lambda} \norm{(h,k,w)}_{\HH}} \\
		& \quad = \inf_{\substack{(\phi_0, \sigma_0, p_0) \in K \\ \norm{(h,k,w)}_{\HH}=1 }} \frac{\norm{(Y(T), Z(T), P(T))}_{\HH}}{C_{\Lambda} \norm{(h,k,w)}_{\HH}}
		\ge \frac{L}{C_{\Lambda}} e^{-Q_2^2}, 
	\end{align*}
	where we used the linearity of $\D\Rcal$ and, since $\Lambda$ is finite-dimensional and thus all norms on it are equivalent, we estimated $\norm{(h,k,w)}_{\VV} \le C_{\Lambda} \norm{(h,k,w)}_{\HH}$.
	
	Then, we can apply Lemma \ref{lem:vessella} and conclude the proof of Theorem \ref{thm:lipstab}. 
\end{proof}

\begin{remark}
    We would like to point out that the constant $C_s$ blows up exponentially as the dimension of $\Lambda$ goes to infinity. 
    Indeed, as one can see from its expression, $C_s$ has a direct exponential proportionality to the constant $C_\Lambda$, which is finite only if $\Lambda$ has a finite dimension and blows up as it becomes larger. 
    Additionally, one can also notice that the dependence of $C_s$ on the final time is even worse, namely $\log(C_s)$ depends exponentially on $T$. 
    This means that, even if we have some kind of Lipschitz stability, one has to be very careful when designing numerical algorithms to approximate the solution.
\end{remark}

\section{Conclusions}
\label{sec:conclusion}

In this paper we consider a phase field model for prostate cancer growth previously introduced and studied in \cite{CGLMRR2019, CGLMRR2021}. 
While improving previous well-posedness results, we rigorously analyse an inverse problem aiming to reconstruct an earlier tumour state starting from a single measurement at a later time. 
We first derive a conditional logarithmic stability estimate and we subsequently prove that the solution to the inverse problem is unique.
Furthermore, by assuming to know a priori that the solution lies in a finite-dimensional subspace, we are also able to obtain an optimal Lipschitz stability estimate, as in Theorem \ref{thm:lipstab}.  
A stability estimate of this kind generally opens the possibility of using iterative reconstruction algorithms, such as the Landweber scheme, to approximate numerically the solution to the inverse problem. 
Indeed, under the validity of a Lipschitz stability estimate and the Fr\'echet-differentiability of the forward map (or more generally of the so-called tangential cone condition), it is shown that such iterative methods are locally convergent \cite{DQS2012, HNS1995, kaltenbacher:neubauer:scherzer, HR2017}.
Therefore, they should provide faithful reconstructions of the earlier tumour state (i.e., the initial conditions of the forward problem), at least if the final time $T$ is small.
On the other hand, due to the severe ill-posedness of the non-linear inverse problem, as shown in Theorem \ref{thm:lipstab}, the stability constants become doubly exponentially large with $T$.
This means that numerical analysis of this problem for large final times $T$ is extremely challenging.
However, in our specific cancer growth application, it is important to be able to efficiently treat even the case of long time horizons, to better understand how the tumour has developed from a small mass to the situation observed in the medical images at diagnosis.
This information on tumour origin and early dynamics holds important value to estimate prognosis and survival, anticipate potential complications derived from disease progression, and contribute to the planning of effective treatments for each individual patient \cite{Lorenzo2022_review,chaudhuri2023predictive,eyupoglu2013surgical,jungk2019location,barrett2019pi,ali2024tale}.
For this reason, the objective of a future part of our project is to carry out an in-depth simulation study to validate our analytical findings and propose efficient reconstruction methods.

\bigskip
	
\noindent\textbf{Acknowledgements.}
C. Cavaterra, M. Fornoni and E. Rocca have been partially supported by the MIUR-PRIN Grant 2020F3NCPX ``Mathematics for industry 4.0 (Math4I4)''. 
C. Cavaterra has been partially supported by the MIUR-PRIN Grant 2022  
	``Partial differential equations and related geometric-functional inequalities''. 
 E. Rocca also acknowledges the support of Next Generation EU Project No.P2022Z7ZAJ (A unitary mathematical framework for modelling muscular dystrophies).
The research of E. Beretta has been supported by the Project AD364 - Fund of NYU Abu Dhabi.
 The research of C. Cavaterra is part of the activities of ``Dipartimento di Eccellenza 2023-2027'' of Universit\`a degli Studi di Milano.
C. Cavaterra, M. Fornoni and E. Rocca are members of  
	GNAMPA (Gruppo Nazionale per l'Analisi Matematica, la Probabilit\`a e le loro Applicazioni)
	of INdAM (Istituto Nazionale di Alta Matematica).
G. Lorenzo acknowledges the support of a fellowship from ”la Caixa” Foundation (ID 100010434). The fellowship code is LCF/BQ/PI23/11970033. 


\footnotesize

\end{document}